\documentclass[reqno,russian,english]{amsart}
\title[Long-range DLA II]{One-dimensional long-range Diffusion Limited
                          Aggregation II: the transient case}

\author{Gideon Amir}
\author{Omer Angel}
\author{Gady Kozma}

\usepackage{amsmath,amsthm,amssymb,amsfonts,url}
\usepackage{varioref}
\usepackage{calrsfs}
\usepackage{pst-all}
\usepackage{sidecap}
\definecolor{aleacolour}{rgb}{0.09,0.32,0.44} 
\usepackage[pdfborder={0 0 0},pdfborderstyle={},colorlinks=true,linkcolor=aleacolour,citecolor=aleacolour,urlcolor=blue]{hyperref}

\usepackage[medium]{titlesec}
\titleformat{\section}
{\normalfont\Large\bfseries}{\thesection}{1em}{}
\titleformat{\subsection}
{\normalfont\bfseries}{\thesubsection}{1em}{}
\titlespacing*{\section}{0pt}{3.5ex plus 1ex minus .2ex}{2.3ex plus .2ex}
\titlespacing*{\subsection}{0pt}{3.25ex plus 1ex minus .2ex}{1.5ex plus .2ex}

\newtheorem{thm}{Theorem}
\newtheorem{lemma}[thm]{Lemma}
\newtheorem{coro}[thm]{Corollary}
\newtheorem{prop}[thm]{Proposition}

\theoremstyle{definition}
\newtheorem{defn}{Definition}
\newtheorem*{defn*}{Definition}
\theoremstyle{remark}
\newtheorem*{rem*}{Remark}

\newcommand{\thmref}[1]{Theorem~\ref{T:#1}}

\newcommand{\eps}{\varepsilon}
\renewcommand{\P}{\mathbb{P}}
\newcommand{\E}{\mathbb{E}}
\newcommand{\R}{\mathbb{R}}
\newcommand{\Z}{\mathbb{Z}}
\newcommand{\N}{\mathbb{N}}
\newcommand{\1}{\mathbf{1}}

\newcommand{\cB}{\mathcal{B}}

\newcommand{\cF}{\mathcal{F}}
\newcommand{\F}{\mathcal{F}}
\newcommand{\Gg}{\mathcal{G}}

\newcommand{\hS}{\widehat{S}}
\newcommand{\hT}{\widehat{T}}
\newcommand{\htau}{\widehat{\tau}}

\DeclareMathOperator{\capa}{Cap}

\DeclareMathOperator{\diam}{diam}

\DeclareMathOperator{\Med}{Med}

\usepackage[T2A]{fontenc}
\DeclareRobustCommand{\cyrtext}{%
  \fontencoding{T2A}\selectfont\def\encodingdefault{T2A}}
\DeclareRobustCommand{\textcyr}[1]{\leavevmode{\cyrtext #1}}

\begin{document}

\begin{abstract}
  We examine diffusion-limited aggregation for a one-dimensional random
  walk with long jumps. We achieve upper and lower bounds on the growth
  rate of the aggregate as a function of the number of moments a single
  step of the walk has.  In this paper we handle the case of transient
  walks.
\end{abstract}

\maketitle

\section{Introduction}

Diffusion-limited aggregation (DLA for short) is a random growth model
constructed as follows.  Start with a single particle in space.  Each
subsequent particle performs a random walk ``from infinity'', until it hits
any previous particle.  It is then frozen and added to the aggregate at the
last site it visited before hitting the aggregate.  Precise definitions are
included below.  We refer the reader to part I \cite{partI} of this project
for a history of the subject, and to part III \cite{partIII} for some
additional results.

The topic of these papers is one-dimensional long-range DLA.  The
following theorem was stated in part I (precise definitions are included
in \S\ref{sec:defs}).  


\begin{thm}\label{T:all}
  Let $R$ be a symmetric random walk with step distribution satisfying
  $\P(|R_1-R_0|=k) = (c+o(1)) k^{-1-\alpha}$. Let $D_n$ be the diameter of
  the $n$ particle aggregate. Then almost surely:
  \begin{enumerate}
  \item \label{enu:begrec}If $\alpha>3$, then $n-1 \le D_n \le Cn+o(n)$,
    where $C$ is a constant depending only on the random walk.
  \item If $2<\alpha\le 3$, then $D_n=n^{\beta+o(1)}$, where $\beta =
    \frac2{\alpha-1}$.
  \item \label{enu:endrec}If $1<\alpha<2$ then $D_n=n^{2+o(1)}$.
  \item \label{enu:1/3to1}If $\frac{1}{3}<\alpha<1$ then
    \[
    n^{\beta+o(1)} \le D_n \le n^{\beta'+o(1)}
    \]
    where $\beta=\max(2,\alpha^{-1})$ and $\beta' =
    \frac{2}{\alpha(2-\alpha)}$.
  \item \label{enu:le1/3}If $0<\alpha<\frac{1}{3}$ then $D_n=n^{\beta+o(1)}$, where
    $\beta=\alpha^{-1}$.
  \end{enumerate}
\end{thm}

\begin{figure}
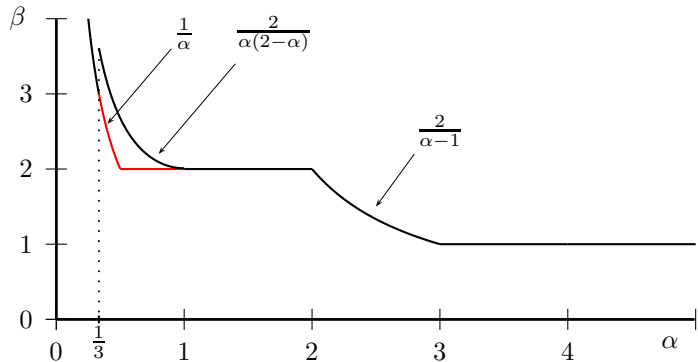

    \psset{xunit=17mm,yunit=10mm}
    \pspicture(-.25,-.4)(5.25,4)
    \psaxes(0,0)(4.99,3.99)
    \rput(-.3,4){$\beta$} \rput(4.8,-.3){$\alpha$}
    \psline(4,1)(5,1)
    \psline(3,1)(4,1)
    \psplot{2}{3}{2 x 1 sub div}
    \psline(1,2)(2,2)
    \psplot{.25}{.3333}{1 x div}
    \psplot[linecolor=red]{.3333}{1}{1 x div 2 max}
    \psplot{.3333}{1}{2 x div 2 x sub div .01 add}
    \psline[linestyle=dotted](.3333, 3.6)(.3333,.1)
    \psline(.3333,-.05)(.3333,.05)
    \rput(.3333,-.3){$\tfrac{1}{3}$}
    \psset{linewidth=0.01}
    \rput(1.7,3.8){$\frac2{\alpha(2-\alpha)}$} \psline{->}(1.4,3.4)(.8,2.2)
    \rput(1,3.8){$\frac1{\alpha}$} \psline{->}(0.9,3.6)(.4,2.6)
    \rput(3,2.5){$\frac2{\alpha-1}$} \psline{->}(2.8,2.2)(2.5,1.5)
    \endpspicture
    \caption{If the random walk $R$ has $\alpha$ finite moments, then the
    diameter of the resulting $n$-particle aggregate grows as $n^\beta$.}
  \label{fig:graph}
\end{figure}

See Figure~\ref{fig:graph} for the various regimes described in
\thmref{all}.  In the regime $\alpha\in(\frac{1}{3},1)$, our results do not
allow us to determine precisely the rate of growth, we conjecture that the
lower bound gives the correct behaviour, and have therefore indicated it in
the figure in red.  Part I focused mainly on the recurrent case, namely
$\alpha>1$.  In this paper we focus on the transient
case, namely $\alpha<1/3$ and $\alpha\in[1/3,1)$.

The most interesting feature of the graph is of course the phase
transitions: at $3$, $2$ and in at least one unknown point in
$\left[\frac{1}{3},1\right]$, probably at $\frac{1}{2}$.  We discussed the
phase transitions at 3 and 2 in the introduction of part I quite
thoroughly, so we will not repeat this here.  Let us reiterate one point
already made in part I nonetheless: it seems as if there is no change in
the behaviour when passing from the recurrent regime to the transient
regime (at $\alpha=1$).  If our conjecture is correct then the growth of
the aggregate is uniformly $n^{2+o(1)}$ throughout the interval
$[\frac{1}{2},2]$ and one cannot tell from the aggregate the difference
between recurrent and transient random walks.  Note that these are not even
quite the same processes on both sides --- in the transient case one needs
to condition on the particles hitting the aggregate, which makes the
``gluing measure'' quite different.  Even if our conjecture is false, our
upper bound still shows that the transition is smooth, hence the phase
transition at $\alpha=1$, if it exists at all, is very tame.

\subsection{Proof ideas}

The key is that the growth of the aggregate does not come by gradual
accretion but rather by isolated large jumps (in fact this holds for all
$\alpha < 3$, see part I for details).  Hence to understand the behaviour
of $A_n$ one needs to consider events whose probability is of the order of
$1/n$ --- about one such event will happen in the construction of $A_n$,
and this one rare event will dominate the growth of the aggregate.

We will show below that the probability that the aggregate $A_n$ grows by
at least $m$ in one step is approximately
\[
\frac{n m^{-\alpha}}{\capa A_n}
\]
where $\capa A_n$ is the \emph{capacity} of the aggregate with respect to
our random walk (see \S\ref{sec:not} for precise definitions).  The reason
that the capacity appears is that in order to define what does it mean for
a particle to perform ``a random walk from infinity'' we have to condition
on the particle hitting the set.  This conditioning gives the factor of the
capacity in the denominator, and is also the reason that the transient case
$0 < \alpha < 1$ is more difficult than the recurrent case discussed in
part I.  Understanding the capacity requires more detailed knowledge about
the structure of the aggregate.

If the largest single jump gives the diameter, then we should have
$\diam A_n\simeq m$ for the $m$ that corresponds to probability
$\frac1n$. Solving
$n m^{-\alpha}/\capa A_n = 1/n$ gives 
\[
\diam A_n \simeq \left(\frac{n^2}{\capa A_n}\right)^{1/\alpha}
\]
so an upper bound for the capacity gives a lower bound for the diameter and
vice versa.  There are two simple upper bounds for the capacity of a
general finite set.  The first is $\capa A_n \le n$ which gives $\diam A_n
\gtrsim n^{1/\alpha}$.  The second bound is $\capa A_n \le c(\diam
A_n)^{1-\alpha}$ (see Lemma~\ref{lem:capconv}) which gives $\diam A_n
\gtrsim n^2$.  This gives the lower bounds for $\alpha<1$ in \thmref{all}.

To test how good these two lower bounds are, consider a discrete, self-similar Cantor set.  It turns out that up to
constants, such a Cantor set has maximal possible capacity with respect to
its diameter and number of particles  
(see Theorem~\ref{thm:capa_Cantor} on page~\pageref{thm:capa_Cantor}).
This is surprising when one compares this to the structure of the harmonic
measure on the fractal which is complicated and involves various exponents
which, generally speaking, are not known. See Carleson \cite{C85} and
Makarov \cite{M98} for analysis of Cantor-like sets, and the beautiful
general results of Makarov \cite{M89} and Jones and Wolff
\cite{JW88}. See also \cite{B99}.

Now, if the aggregate were to behave similarly, then it would have nearly
maximal possible capacity which would lead to the minimal possible
diameter.  This would imply that the third phase transition is at
$\frac{1}{2}$.  If, on the other hand, the aggregate has a dense core with
a few additional far away particles which do not affect the capacity
significantly, we would get a much lower capacity, and hence higher
diameter.  In that case, the third phase transition could even be at $1$.

One regime where we can make this ``Cantor set vs.\ heavy core'' dichotomy
precise, and so derive matching upper bounds on the diameter, is when the
walk has less than $\frac{1}{3}$ moments.  In this regime the walk makes
enormous jumps --- by time $n$ of the aggregation we expect a jump of size
at least $n^3$, so a particle will be glued at large distance from the rest
of the aggregate.  Assume such a jump has happened, and call the resulting
particle $b_1$.  Let $b_2$ be the first time a particle coming from
infinity hits $\{b_1\}$.  Continuing this way we identify a subset $B$ in
the aggregate consisting of $b_1$ and its descendents.  A calculation then
shows that the subset $B$ grows very similar to a DLA --- the affect of the
rest of the aggregate is only a time-change which does not affect the
actual shape too much.

Since such large jumps occur at many scales, we think of DLA as a
stochastic Cantor set, where remote parts are identically distributed,
(though not identical, as in the usual Cantor set).  The crux of the
argument is in making this precise and proving this.  We show that the
probability that the existing part $A\setminus B$ affects the growth of the
subset $B$ or vice versa \emph{at all} is quite small.  This is not true
for $\alpha>\frac{1}{3}$.  For such $\alpha$, the amount of interaction
between the two parts, while small, is not zero, and we do not know how to
control it.

Of course, the actual proof entails a few complications. Since we lose
constants in various places, it is not enough to divide $A_n$ into 2 parts
and show that they are roughly identically distributed.  We therefore
divide $A$ into $\log n$ parts, show that (with high probability) they are
identically distributed and that the capacity of the union is close to the
sum of their capacities (again, because they are far away) this gives the
formula
\[
\capa A_n \simeq (\log n)\capa A_{n/\log n}.
\]
Now the loss of constants is immaterial and we get $\capa A_n \ge
n^{1-o(1)}$, and so the upper bound on the diameter.  Another complication
is that it seems difficult to work with the expectation of $\sum 1/\capa
A_n$ (this sum turns out to be the relevant quantity to study) because it
is difficult to control unusual events which might make this sum large.
Thus instead of expectations we work with medians and quantiles, and you
will see them strewn all over the proof.

The approach could also give improved bounds for some $\alpha>\frac{1}{3}$.
In fact, the condition required for the argument to apply is roughly that
$n^2(\diam A_n)^{\alpha-1}\lesssim 1$ (see Lemma~\ref{L:TSnu} on page
\pageref{L:TSnu} below).  This means that if $\diam A_n \gtrsim
n^{2/(1-\alpha)}$ then our argument will go through, leading to a
contradiction as the capacity grows linearly while the diameter is still
large.  In short, for $\alpha\in(\frac{1}{3},1)$ it is possible to show
that $\diam A_n \le n^{2/(1-\alpha)+o(1)}$ using essentially the same
argument as we use for $\alpha < \frac13$.  This estimate is better than
our stated upper bound in the interval $[\frac{1}{3},
\frac{1}{2}(3-\sqrt{5})]$ (approximately $0.38$).  Nevertheless, we will
not prove this estimate, as we do not believe it contributes additional
understanding to the problem.  We remark, though, that adding this extra
result would remove the discontinuity in the graph of our upper bound,
making it continuous (but not monotone).


\subsection{What should you read?}

This paper is more technical than part I, with a particular emphasis on the
case of $\alpha<\frac{1}{3}$.  If you are only interested in a tasting of
the ideas used, your best option is probably to switch to part I and read
\S{}3 there which handles one particular transient random walk whose
analysis is simpler and more geometric than other cases.

In Section \ref{sec:defs} we derive a general formula for the gluing
measure in the transient case.  Section \ref{sec:01} then discusses the
lower bound $D_n\ge n^{\beta+o(1)}$ which holds for all $0<\alpha<1$.
Finally Section \ref{sec:third} gives the proof for the case of
$\alpha<\frac{1}{3}$.  This paper is best read sequentially as each section
depends on all previous ones.

Finally, note that there is a part III \cite{partIII} which discusses the
infinite aggregate $A_\infty=\bigcup_{n=1}^\infty A_n$.

\subsection{Notations}
\label{sec:not}

For a subset $A\subset \Z$ we denote by $\diam A$ the diameter of $A$,
namely $\max(A) - \min(A)$.  If $x\in\Z$ we will denote by
$\rho(x,A)$ the point-to-set distance, namely $\min_{y\in A} |x-y|$.

Throughout this paper we denote by $R=(0=R_0, R_1, \dotsc)$ a random walk
on $\mathbb{Z}$.  We denote by $p_{x,y} = p_{0,y-x}$ the probability that
$R$, when starting from $x$ will go to $y$ in the first step.
For a given random walk $R$ and set $A$, let $T_A$ be the hitting
time of $A$, defined as
\[
  T_A = \inf \{ n>0 \text{ s.t. } R_n\in A \}.
\]
Note that $T_A>0$ even if the random walk starts in $A$.  Denote by
$E_A(x)$ the escape probability from $A$ i.e.\ $\P_x(T_A=\infty)$ (where
$\P_x$ denotes the law of the random walk started at $x$).  Denote by
$E_A^*(x)$ the escape probability for the reversed random walk (we will
usually consider symmetric random walks, in which case $E_A=E^*_A$).
Define the {\bf capacity} of a set by
\[
\capa A = \sum_{a\in A}E_A^*(x).
\]
See \S\ref{sec:capgen} for a detailed discussion of the capacity.

Denote by $G(x,y)$ the Green function of $R$ defined by
\[
 G(x,y)=\sum_{n=0}^\infty \P_x(R_n=y)
\]
Since $G$ is translation invariant we will often abbreviate $G(x,0)$ as
$G(x)$ and then $G(x,y)=G(x-y)$.  It turns out that regular behaviour of
the Green function is the most important property for our analysis.

By $C$ and $c$ we denote constants which depend only on the law of the walk
$R$ but not on the other parameters involved.  The same holds for constants
hidden within the notations $O$ and $o$.  In particular, we will usually
\emph{not} use them again the way they were used in theorem \ref{T:all}
that is, as random constants.  Generally $C$ and $c$ might take different
values at different places, even within the same formula. Normally, $C$ will 
pertain to constants which are ``big enough'' and $c$ to constants which
are ``small enough''.  By $X\approx Y$ we mean $cX<Y<CX$.  By $X\ll Y$ we
mean that $X=o(Y)$.  $\lfloor x\rfloor$ will denote the integer value of
$x$. By $\simeq$ and $\lesssim$ we do not mean anything in particular;
these are only used for the heuristic discussion in the introduction.

\section{Limits of gluing measure}
\label{sec:defs}

The random walks we consider will usually satisfy the following:

\begin{defn*}
  We say that a symmetric random walk is an $\alpha$-walk ($\alpha<1$) if
  $p_{0,x} \approx |x|^{-1-\alpha}$.
\end{defn*}

(As usual $\approx$ means ``bounded between two constants''.) For such
walks the Green function has the following behaviour:

\begin{lemma}\label{lem:G}
  Let $R$ be an $\alpha$-walk.  Then its corresponding corresponding Green
  function satisfies
  \begin{equation}
  G(x)\approx|x|^{\alpha-1}.\label{eq:Gxalpha}
  \end{equation}
\end{lemma}

\begin{proof}
  We base our proof on a result of Bass \& Levin \cite{BL02}.  Let us
  recall the statement of Theorem~1.1 ibid.  Bass \& Levin do not assume
  that the walk is translation invariant, and their result holds in any
  dimension, though we specialize to $d=1$.  Given numbers $w_{xy}$
  satisfying $w_{xy}\approx |x-y|^{-1-\alpha}$ and $w_{xy}=w_{yx}$, define
  a Markov chain $R$ by its transition probabilities
  \[
  \P_x(R_1=y) = \frac{w_{xy}}{\sum_z w_{xz}}.
  \]
  We may therefore take $w_{xy} = p_{x,y}$.  A minor inconvenience is that
  Bass \& Levin assume that $w_{xx}=0$. Let us therefore assume it for our
  walk $R$ for a while, and remove this assumption in the end. Their
  results then state that
  \begin{equation}\label{eq:BL}
    \P_x(R_n = y) \approx \min(n^{-1/\alpha}, n |x-y|^{-1-\alpha})
  \end{equation}
  with the exception that the lower bound does not hold for $n=1$ and
  $x=y$ (again, because $w_{xx}=0$).

  Summing (\ref{eq:BL}) over $n$ immediately gives (\ref{eq:Gxalpha}).
  Thus we need only remove the restriction $p_{x,x} = 0$.  But this only changes
  $G$ by a constant, i.e.\ if for some $R$ we have $p_{x,x} > 0$ then
  we can define a walk $R'$ using
  \[
  \P(R'_1=x) = \begin{cases}
    \frac{p_{0,x}}{1-p_{x,x}} & x\ne 0, \\
    0 & x=0.
  \end{cases}
  \]
  and get that the corresponding Green function, $G'$, satisfies
  $G'(x) = G(x)/p_{x,x}$. This finishes the lemma.
\end{proof}

Let us remark that the results of Bass \& Levin are not restricted to the
transient case, they hold also for the recurrent case.  However, it is not
as straightforward to get information on the harmonic potential from the
estimate of the heat kernel, as it is for the Green function in the
transient case.

While on this topic we might remark that a number of results in the
non-reversible settings are known.  The results of Williamson \cite{W68}
are quite general, essentially requiring only that the walk $R$ is in the
domain of attraction of an $\alpha$-stable process, but they are restricted
to $\frac{1}{2}<\alpha<1$.  Le Gall and Rosen have some results for general
$\alpha$ \cite[Proposition 5.2]{GR91}.


Let $A\subset\Z$ be some set.  For a point $x\notin A$, we are interested
in the event $\{R_{T_A-1} = x\}$, that is the event that $x$ is the last
point visited before hitting $A$, and so is glued.  Since the random walks
we are interested in are transient we must condition on $T_A<\infty$.
Define
\[
\mu_y(x,a) = \P_y \bigl( R_{T_A}=a, R_{T_A-1}=x \, \vert\,  T_A<\infty \bigr).
\]
We are interested in the limit
\[
\mu(x,a) = \mu(x,a;A) = \lim_{y\to\infty}\mu_y(a,x),
\]
if it exists, as well as its integrated version
\[
\mu(x) = \mu(x;A) = \sum_{a\in A} \mu(x,a).
\]
The existence of these limits is strongly related to the existence of the
harmonic measure from infinity, which is defined as the limit of $\sum_x
\mu(x,a)$ (see Lemma~\ref{L:mu_trans} below).

It is possible that the harmonic measure does not exist. An example can be
constructed by examining a random walk supported on a very sparse set, for
example let $\P(R_1=2^{2^n}) \approx 1/n^2$ and $0$ otherwise. As this is
somewhat off-topic we will not provide a proof that this is indeed an
example. Instead, the following theorem characterises walks for which the
harmonic measure exists, and relates the harmonic measure to the escape
probabilities.

\begin{thm}\label{T:transient}
  The following are equivalent for a transient irreducible random walk on
  $\Z$:
  \begin{enumerate}
  \item\label{enu:allA} For any finite $A$, the harmonic measure exists and
    equals the normalized escape probabilities.
  \item\label{enu:someA} The harmonic measure exists for some finite $A$
    with $|A|>1$,
  \item\label{enu:GxGx+a} For any $a$ the Green function satisfies
    $\lim_{|x|\to\infty} \frac{G(x+a)}{G(x)} = 1$.
  \end{enumerate}
\end{thm}

When the conditions of this theorem hold we say that $R$ has a harmonic
measure.

Recall that the escape probabilities are denoted by $E_A(x) =
\P_x(T_A=\infty)$, and that this definition is non-trivial also for
starting points in $A$. Recall also that $E^*_A(a)$ is the analogous
probability for the reversed random walk (for symmetric random walk the two
are the same). The capacity of a set $A$ is defined as $\capa(A) =
\sum_{a\in A} E^*_A(a)$.

Note that the result is stated also for asymmetric walks even though we are
mostly interested in the symmetric case. Further it can be easily
generalised to the case of $\Z^d$ --- the only difference being that in
Clause \ref{enu:someA} we must require that $A$ is not a subset of any
affine subspace of $\R^d$.

Part of the proof (\ref{enu:GxGx+a}$\implies$\ref{enu:allA}) follows the
proof of P26.2 of \cite{S76} which does the same for random walks in
$\Z^3$.

\begin{proof}[Proof of \thmref{transient}]
  Let $\Pi$ be the $A$-indexed matrix, with $\Pi(a,b)=\P_a(R_{T_A}=b)$. Let
  $G(x-\cdot)$ be the row vector $\{G(x-a)\}_{a\in A}$.  We use the
  following identity for the unnormalized hitting measure
  $H_A(x,a) = \P_x(R_{T_A}=a)$

  \begin{prop}\label{prop:HGPI}
    $\displaystyle H_A(x,a) = \sum_{z\in A} G(x-z) (I-\Pi)(z,a)$.
  \end{prop}

  A proof can be found in \cite[P25.1]{S76}. For the reader's convenience,
  let us indicate a simple probabilistic argument (Spitzer's proof is more
  analytic in nature).  Examine the random walk restricted to $A$: The
  probability to hit $A$ at $a$ is the expected number of times $a$ is
  visited minus the number of visits that are not the first visit to $A$.
  The former is $G(x,a)$; the latter can be partitioned according to the
  location of $z$: the previous visit to $A$.

  \ref{enu:GxGx+a}$\implies$\ref{enu:allA}: By Proposition~\ref{prop:HGPI}
  and our assumption \ref{enu:GxGx+a},
  \[
  \lim_{|x|\to\infty}\frac{H_A(x,a)}{G(x)}=\lim_{|x|\to\infty}\sum_{z\in
    A}\frac{G(x-z)}{G(x)}(I-\Pi)(z,a)
  \stackrel{\textrm{\ref{enu:GxGx+a}}}{=}
  \sum_{z\in A}(I-\Pi)(z,a).
  \]
  The sum of column $a$ in $I-\Pi$ is $E^*_A(a)$, so 
  \[
    \lim_{|x|\to\infty} \frac{H_A(x,a)}{G(x)}  =
    E_A^*(a)\qquad\forall a\in A.
  \]
  Summing over $a$ gives
  \begin{equation}\label{eq:hit_prob}
    \P_x(T_A<\infty) = G(x) \cdot \capa(A) \cdot (1+o(1)).
  \end{equation}
  Which gives the required:
  \begin{align*}
    \lim_{|x|\to\infty}\P_x(R_{T_A}=a
    \,|\,T_A<\infty)&=
    \lim_{|x|\to\infty}\frac{H(x,a)}{\P_x(T_A<\infty)}\\
    &= \lim_{|x|\to\infty}\frac{H(x,a)}{G(x)\capa(A)(1+o(1))}
    = \frac{E_A^*(a)}{\capa(A)}.
  \end{align*}

  Since \ref{enu:allA}$\Rightarrow$\ref{enu:someA} is obvious, it remains
  to show that \ref{enu:someA} implies \ref{enu:GxGx+a}. For any finite
  $A$, $I-\Pi$ is strictly diagonally dominant and therefore invertible (by
  the Levy-Desplanques theorem, see e.g.\ \cite[Theorem
  6.1.10]{HJ85}). Thus if the harmonic measure exists (which means, by
  definition, that $\lim_{|x|\to\infty}\P(R_{T_A}=a\,|\,T_A<\infty)$
  exists), then
  \begin{align*}
    \lim_{|x|\to\infty}\frac{G(x-a)}{\P_x(T_A<\infty)}
    &=\lim_{|x|\to\infty}\sum_{z\in A}
    \frac{H_A(x,z)}{\P_x(T_A<\infty)}(I-\Pi)^{-1}(z,a)\\
    &=\sum_{z\in A}\mu(z)(I-\Pi)^{-1}(z,a)
  \end{align*}
  and in particular the limit on the left-hand side exists. Hence also the
  limit of $G(x-a)/G(x-b)$ exists, for all $a,b\in A$.

  Assume without loss of generality that $0,b\in A$ and then the limit of
  $G(x-b)/G(x)$ exists.  Since the limit is the same for positive and
  negative $x$, and since $0 < G(x) \le G(0)$ is bounded, the limit must be
  1 --- if it were smaller than 1 then $G$ would diverge exponentially in
  the direction of $b$ while if it were larger than 1 then $G$ would
  diverge exponentially in the other direction.

  Clearly $\frac{G(x-b)}{G(x)} \to 1$ implies that $\frac{G(x-kb)}{G(x)}
  \to 1$ for any $k$.  To show that $\frac{G(x+a)}{G(x)} \to 1$ for general
  $a$ we need irreducibility for the first time.  Let $T$ be the first time
  at which $R_T\equiv a \pmod b$, and note that $T$ is a.s.\ finite. Fix
  $\eps>0$ and let $N=N(a,\eps)$ be such that $\P(|R_T-a|>N)<\eps$. Take
  $M=M(a,\eps,N)$ large enough that whenever $|kb|<N$ and $x>M$ we have
  $G(x+kb) > (1-\eps)G(x)$.

  By the Markov property at $T$, for $|x|>M$.
  \begin{align*}
    G(x+a)
    &= \sum_{y\equiv a \textrm{ mod } b} \P(R_T=y)G(x+a-y) \\
    &\ge \sum_{\substack{|y|<N \\ y\equiv a\textrm{ mod } b}} \P(R_T=y)(1-\eps)G(x)\\
    &\ge (1-\eps)^2 G(x).
  \end{align*}
  Since $\eps$ was arbitrary, $\liminf_x G(x+a)/G(x) \ge 1$. Since the same
  holds for $-a$ this completes the proof.
\end{proof}

We now come to the key connection between escape probabilities and the DLA
gluing measure.

\begin{lemma} \label{L:mu_trans}
  If the harmonic measure for a transient random walk $R$ exists, then
  the limit $\mu=\lim_y \mu_y$ exists and
  \begin{equation} \label{eq:repmu}
    \mu(x,a) = \frac{p_{x,a} E^*_A(x)}{\capa(A)}.
  \end{equation}
\end{lemma}

We remind the reader that $p_{x,a}$ is the probability of making a single
step from $x$ to $a$.

\begin{proof}
  Define $F(z) = \P_z(T_x<T_A)$ to be the probability that $x$ is hit from
  $z$ before $A$ (and in particular $T_x<\infty$). Start a random walk at
  $y$. One can partition visits to $x$ according to the previous visited
  site of $A\cup\{x\}$ to get
  \[
  G(y,x) = G(y-x) = F(y) + G(y-x)F(x) + \sum_{a\in A} G(y-a)F(a).
  \]
  On the other hand, the expected number of visits to $x$ before visiting
  $A$ is given by $\frac{F(y)}{1-F(x)}$. Let $M_{x,a}$ be the event that
  the random walk hits $A$ by a step from $x$ to $a$, then
  \begin{align*}
    \P_y(M_{x,a}) &= \frac{F(y)p_{x,a}}{1-F(x)} \\
    &= p_{x,a} \cdot
       \left( G(y-x) - \frac{1}{1-F(x)}\sum_{a\in A} G(y-a)F(a) \right).
  \end{align*}
  Using \eqref{eq:hit_prob} to approximate $\P_y(T_A<\infty)=G(y)\capa(A)(1+o(1))$ and condition
  \ref{enu:GxGx+a} of \thmref{transient} to eliminate the ratio of Green
  function values as $y\to\infty$ one finds
  \[
  \P_y(M_{x,a}\,|\,T_A<\infty) = \frac{p_{x,a}}{\capa(A)} \cdot
          \left( 1 - \frac{1}{1-F(x)}\sum_{a\in A} F(a)\right)(1+o(1))
  \]
  (here $o(1)$ is as $|y|\to\infty$).
  
  Next we relate this to the reversed escape probability from $x$. Run a
  reversed random walk from $x$. Paths that hit $A$ at $a$ can be broken
  into some number of loops where the walk returns to $x$ followed by a
  path from $x$ to $a$ that avoids $A\cup\{x\}$. The reverse of a loop
  rooted at $x$ is just a loop; the reverse of the last segment is one of
  the terms contributing to $F(a)$.  Thus we have
  \[
  \P^*_x(\text{hit $A$ at $a$}) = \frac{F(a)}{1-F(x)}.
  \]
  Summing over $a\in A$ gives the probability of hitting $A$ and so
  \[
  \P_y(M_{x,a}\,|\,T_A<\infty)= \frac{p_{x,a}}{\capa(A)} E^*_A(x)(1+o(1)).
  \qedhere
  \]
\end{proof}




We are finally ready to give the formal definition of the DLA process for
transient random walks:

\begin{defn}
  Let $R$ be a transient random walk with harmonic measure on $\Z$. The \emph{DLA process with respect to
    $R$} is a Markov chain of random sets $A_1=\{0\},A_2,\dotsc$ such that
  for any finite $A \subset \Z$, and $x\in\Z\setminus A$ and any $n>0$,
  \[
  \P(A_{n+1}=A\cup\{x\}\,|\,A_n=A) = \sum_{a\in A} \mu(x,a;A)
  \]
  where $\mu$ is defined by (\ref{eq:repmu}).  We denote $D_n=\diam A_n$
  and $\F_n$ the minimal $\sigma$-field generated by $A_1,\dotsc,A_n$.
\end{defn}

\section{Inequalities}
\label{sec:01}

Recall the gluing formula (\ref{eq:repmu}). It turns out that the most
crucial term for understanding the DLA growth is the capacity $\capa(A)$
appearing in the denominator. In \S\ref{sec:capgen} we will prove estimates
for the capacity of any finite set in $\Z$. In \S\ref{sec:capdiam} we will
use these estimates to get first bounds on the diameter of the aggregates.
Henceforth, all walks are assumed to be symmetric, and we will not remind
this fact again.

\subsection{General inequalities for capacity}
\label{sec:capgen}

Before embarking on estimates of the capacity of our aggregate, let
us recall some general facts about this notion. The capacity of a set can be
defined in two equivalent ways. The first (which we used to prove
(\ref{eq:repmu})) is as the sum of the escape probabilities. The second
is a variational definition:
\begin{equation}\label{eq:variational}
\capa(A) = \sup_{\psi\in\mathcal{P}(A)}\frac{1}{\|\psi*G\|_{\infty}},
\end{equation}
where $\mathcal{P}(A)$ is the set of all probability measures on $A$ (or
simply all $\psi:A\to[0,1]$ with $\sum\psi(a)=1$), $G$ is the Green
function of the walk and $\psi*G$ denotes the usual convolution on $\Z$
i.e.
\[
(\psi*G)(x):=\sum_{y}G(x-y)\psi(y).
\]
See \cite[P25.10]{S76} for the equivalence of the two definitions. There is a
unique measure $\psi_0$ satisfying the supremum which is called \emph{the
  equilibrium charge}. It satisfies that $\psi_0 * G$ is in fact constant over
$A$ and is characterized by this fact. In fact, the equilibrium charge is the
normalized escape probabilities, namely
\[
\psi_0(x)=\frac{\P_x(T_A=\infty)}{\sum_{a\in A}\P_a(T_A=\infty)}.
\]
See \cite[chap.\ 25]{S76} for an overview and proof of the equivalence of
the two definitions of $\psi_0$.  We start with an application of capacity to show that
starting far from a set suffices for a good probability of avoiding it.

\begin{lemma}\label{lem:tran_hit}
  Let $R$ be an $\alpha$-walk. For any $\lambda>0$
  \[
  \inf \left\{ E_A(x) : \rho(x,A) > \lambda \diam A \text{ and } x>\max(A)
  \right\} > 0.
  \]
\end{lemma}

\begin{proof}
  By translation invariance and since $E_A(x)$ is monotone in $A$, we may
  assume $A=[-k,0]$ for some $k$ (and then $\diam A=k$).
  Define
  \[
  \varphi(\lambda) = \inf_{k,w>\lambda k}  \left\{ E_{[-k,0]}(w) \right\}.
  \]
  We will first prove $\varphi(\lambda_0)>0$ for some $\lambda_0$ large
  enough. Let $\psi$ be the equilibrium charge of $[-k,0]$, and let
  $p=\psi*G$. By \cite[T25.1]{S76},
  \[
  1-E_{[-k,0]}(w) = \frac{p(w)}{p(0)} \quad\forall w\not\in[-k,0].
  \]
  (Note that $p(\cdot)$ is constant on $[-k,0]$, so $0$ is not singled out
  here.)  Let us first assume that $w\geq\lambda_0 k$ for some $\lambda_0$
  sufficiently large.  Then we can write (using $G(x)\approx
  |x|^{\alpha-1}$, Lemma~\ref{lem:G})
  \begin{align*} p(w) & =\sum_{i=0}^{k}\psi(-i)G(w+i)\leq
    C\sum_{i=0}^{k}\psi(-i)(\lambda_0 k)^{\alpha-1}  \\
    & \leq C\lambda_0^{\alpha-1}\sum_{i=0}^{k}\psi(-i)(1+i)^{\alpha-1}\leq
    C\lambda_0^{\alpha-1}\sum_{i=0}^{k}\psi(-i)G(i)\\
    &=C\lambda_0^{\alpha-1}p(0)
  \end{align*}
  and we see that for $\lambda_0$ sufficiently large, independent of $k$,
  $E_{[0,k]}(w) \geq \frac{1}{2}$, and so $\varphi(\lambda_0) \geq
  \frac{1}{2}$.

  Next, fix some $\lambda>0$. Let $w>\lambda k$, and let $I$ be the
  reflection of $[-k,0]$ about $w$, i.e.\ the interval $[2w,2w+k]$. By
  symmetry we have
  \[\P_{w}(T_{[-k,0]} < T_I) \leq \frac{1}{2}.\]
  However, after $T_I$ (if it is finite), the probability to escape to
  infinity is at least $\varphi(2\lambda)$. Thus we get
  \[ \varphi(\lambda) \geq \frac{1}{2}\varphi(2\lambda).
  \]
  Since $\varphi(\lambda_0)>0$ and $\varphi$ is increasing, this implies
  $\varphi(\lambda)>0$ for all $\lambda>0$.
\end{proof}

\begin{lemma}\label{lem:mincapa}
  For an $\alpha$-walk $R$, the capacity of any finite set $A$ satisfies
  \[
  \capa(A) \geq c |A|^{1-\alpha}.
  \]
  If moreover $A\subset d\Z$, then 
  \[
  \capa(A) \geq c \min\{|A|,(d |A|)^{1-\alpha}\}.
  \]
\end{lemma}

\begin{proof}
  We use the variational definition of the capacity.  Examine the uniform
  probability measure $\psi_{u}$. We get
  \[
  \capa(A)\geq \| \psi_{u}*G \|_{\infty}^{-1}.
  \]
  By $G(x)\approx |x|^{\alpha-1}$ (Lemma \ref{lem:G}) we get
  \begin{align*}
  (\psi_{u}*G)(x) &=
  \sum_{y}G(x-y)\psi(y)\leq\sum_{y}(C|x-y|^{\alpha-1})\cdot\frac{1}{|A|}\\
  &\le\frac{1}{|A|}\sum_{i=1}^{|A|}(Ci^{\alpha-1})
  \le C|A|^{\alpha-1}
  \end{align*}
  and the first claim is proved.  If $A$ is contained in an arithmetic
  progression, the same argument gives
  \[
  (\psi_{u}*G)(x) \leq C|A|^{-1} + C(d |A|)^{\alpha-1},
  \]
  where the first term comes from the contribution of $y=x$.
\end{proof}

\begin{lemma}\label{lem:capconv}
  For an $\alpha$-walk $R$ the capacity of any finite set $A$ satisfies
  \[
  \capa(A)\leq C(\diam A)^{1-\alpha}.
  \]
\end{lemma}

\begin{proof}
  It follows directly from the variational definition
  (\ref{eq:variational}) that $\capa(A)$ is
  increasing in $A$ (see \cite[P25.11 (b)]{S76} for a different but equally
  simple argument). Filling in the holes in $A$ and translating to $0$ we get $\capa(A)\leq\capa([0,\diam A])$. To estimate the
  capacity of an interval we use the definition of capacity as a sum of escape
  probabilities and write
  \[
  \capa([0,n]) \leq 2\sum_{i=1}^{n/2} \P_0(|R_1|>i)
  \]
  i.e.\ we bound the escape probability simply by the probability to exit
  the set in the first step. Since $R$ is an $\alpha$-walk, this implies
  \[
  \capa([0,n]) \leq C\sum i^{-\alpha}\leq Cn^{1-\alpha}.\qedhere
  \]
\end{proof}

To demonstrate that the combination of these two bounds is sharp, let us
prove the result promised in the introduction concerning the capacity of
the Cantor set.  While this is a bit of a digression, some of our results
indicate that the DLA aggregate is self-similar, and is built up of several
copies of smaller aggregates, well separated in $\Z$.  Thus the following
serves as further indication that the capacity of our aggregate $A_n$ is indeed
roughly $\min\{n,(\diam A_n)^{1-\alpha}\}$, and consequently the diameter grows as $n^{\min\{2,1/\alpha\}+o(1)}$ also
for $\alpha\in[1/3,1]$.

\begin{thm}
\label{thm:capa_Cantor}Let $A_{n}$ be the discrete Cantor set \[
A_{n}:=\{ i=0,\dotsc,3^{n}-1:\forall j=0,\dotsc,n-1\left\lfloor i/3^{j}\right\rfloor \not\equiv1\mod3\}.\]
Let $R$ be an $\alpha$-walk. Then \[
\capa A_{n}\geq c\min\{2^{n},3^{n(1-\alpha)}\}/n.\]
Further, if $\alpha\neq-\log_{3}2+1$ then the estimate may be slightly
improved to\[
\capa A_{n}\geq c(\alpha)\min\{2^{n},3^{n(1-\alpha)}\}.\]
\end{thm}

Note that by lemma \ref{lem:capconv} $\capa A_{n}\leq3^{n(1-\alpha)}$
and since $|A_{n}|=2^{n}$ we have $\capa A_{n}\leq2^{n}$. Hence
(for $\alpha$-walks) the theorem is precise up to at most the logarithmic factor $n$.

\begin{proof}
Let $\psi$ be the equilibrium charge over $A_{n}$. Then \[
\frac{1}{\capa A_{n}}=(\psi*G)(x)\quad\forall x\in A_{n}\]
and therefore
\begin{align}
\frac{1}{\capa A_{n}} & =2^{-n}\sum_{x\in A_{n}}(\psi*G)(x)=2^{-n}\sum_{x\in A_{n}}\sum_{y\in A_{n}}\psi(y)G(x-y)=\nonumber \\
 & =2^{-n}\sum_{y\in A_{n}}\psi(y)\sum_{x\in A_{n}}G(x-y)\nonumber\\
 &\leq C2^{-n}\sum_{y\in A_{n}}\psi(y)\sum_{x\in
 A_{n}}|x-y|^{\alpha-1}\label{eq:capaavrg}
\end{align}
where in the last inequality we used Lemma \ref{lem:G}.
Now for every $y\in A_{n}$ it is easy to calculate\begin{align}
\sum_{x\in A_{n}}|x-y|^{\alpha-1} & =\sum_{k=0}^{n-1}\sum_{x:3^{k}\leq|x-y|\leq3^{k+1}}|x-y|^{\alpha-1}\leq\nonumber \\
 & \leq\sum_{k=0}^{n-1}3^{k(\alpha-1)}\#\{ x:|x-y|\leq3^{k+1}\}\leq C\sum_{k=0}^{n-1}3^{k(\alpha-1)}2^{k}\leq\nonumber \\
 & \leq Cn\max(2^{n}3^{n(\alpha-1)},1).\label{eq:alphacrit}\end{align}
Further, if $\alpha\neq-\log_{3}2+1$ then the last inequality may
be strengthened and we get\begin{equation}
\sum_{x\in A_{n}}|x-y|^{\alpha-1}\leq C(\alpha)\max(2^{n}3^{n(\alpha-1)},1)\quad\forall y\in A_{n}\label{eq:alphanoncrit}\end{equation}
Plugging (\ref{eq:alphacrit}) and (\ref{eq:alphanoncrit}) into (\ref{eq:capaavrg})
gives the two parts of the theorem, respectively (remember that $\sum\psi(y)=1$).
\end{proof}

Monotonicity of the capacity implies that adding a point to a set will
increase its capacity. The next lemma gives a nice formula for the exact
increment in capacity. Recall that for a set $A$ and point $x$, $E_A(x)$ is
the escape probability starting at $x$.  We also define $E'_A(x) =
E_{A\cup\{x\}}(x)$, i.e.\ the probability that $x$ is also avoided by the
random walk.

\begin{lemma} \label{lem:escape2}
 Let $A$ be a set with finite capacity with respect to a transient symmetric random
 walk, let $x$ be a point outside $A$, and let $A'=A\cup\{x\}$.
 Then
 \[
   \capa(A')-\capa(A) = E_A(x)E'_A(x)
 \]
 and in particular, $\capa(A')-\capa(A) \in [E'_A(x)^2, E_A(x)^2]$.
\end{lemma}

\begin{proof}
  From the representation of the capacity as a sum of escape probabilities
  we have $\capa(A')-\capa(A) = E'_A(x) - \sum_{a\in A}
  (E_A(a)-E_{A'}(a))$.  Since the event of escaping from $A'$ is contained
  in the event of escaping from $A$, for $a\in A$ we have,
  \[
  E_A(a)-E_{A'}(a) = \P_a(T_x<\infty, T_A = \infty).
  \]
  Random walk paths that escape from $A$ and hit $x$, a.s.\ visit $x$ only
  finitely many times. Breaking each path at the last visit to $x$ we have
  \[
  E_A(a)-E_{A'}(a)
  = \sum_{\substack{\gamma:a\to x\\ \text{outside A}}}
         \P(\gamma) \P_x(T_{A'} = \infty)
  = E'_A(x) \sum_{\substack{\gamma:a\to x\\ \text{outside A}}} \P(\gamma)
  \]
  where $\P(\gamma)$ denotes the probability that the random walk follows the
  path $\gamma$ until its end. 
  By reversing each path $\gamma$:
  \[
  E_A(a)-E_{A'}(a) = E'_A(x) \P_x(R\text{ hits $A$ at $a$}),
  \]
  since the sum is just the probability that the walk starting at $x$
  hits $A$ at $a$. Summing over $a\in A$:
  \[
  \capa(A')-\capa(A) = E'(x) - E'_A(x) \P_x(T_A<\infty)
  = E'_A(x) E_A(x). \qedhere
  \]
\end{proof}

\subsection{Capacity and DLA}

\begin{lemma}\label{lem:capabalance}
  Let $R$ be an $\alpha$-walk. Let $A_{n}$ be a corresponding DLA
  process. Then there exists some $\delta>0$ such that\[
  \P\Big( \capa(A_{n+1})-\capa(A_{n}) > \delta \,|\, \cF_n \Big)
  \geq \frac{cn}{\capa(A_{n})^{1/(1-\alpha)}}.\] Here $c$ and $\delta$
  may depend on the walk, but not on $n$.
\end{lemma}

\begin{proof}
  Since $A_n$ is an arbitrary set of $n$ elements, let us denote it by $A$
  and its capacity by $\kappa$. For some $\delta$ to be determined, let
  $S_\delta = A\cup \{x : E_A(x) E'_A(x) \le \delta\}$ be the set of points
  whose addition to $A$ will increase the capacity by less than $\delta$.
  The first step is to show that $S_\delta$ is small, in the sense that at
  most half the points in any interval of length at least $ \lambda
  \kappa^{1/(1-\alpha)}$ are in $S_\delta$, where $\lambda$ is some
  constant to be determined.  The scale $\kappa^{1/(1-\alpha)}$ is chosen
  so that the capacity of the whole interval is comparable to $\kappa$.

  To see this, observe first that $E_A(x) E'_A(x)$ is decreasing in $A$.
  Let $D=\{k_1,\dotsc,\linebreak[4]k_{|D|}\}$ be an arbitrary subset of
  $S_\delta$. Then
  \begin{align}
  \capa(D)&\le\capa(A\cup
  D)\nonumber\\
  &=\capa(A)+\sum_{i=1}^{|D|}\capa(A\cup\{k_1,\dotsc,k_{i}\})-\capa(A\cup\{k_1,\dotsc,k_{i-1}\})\nonumber\\
  &=\kappa  + \sum_{i=1}^{|D|}
  E_{A\cup\{k_1,\dotsc,k_{i-1}\}}(k_i)E_{A\cup\{k_1,\dotsc,k_i\}}(k_i)
\le \kappa + |D|\delta\label{eq:capD}
\end{align}
  Fix $d=\kappa^{\alpha/(1-\alpha)}$. For an interval $I$ of length
  $\lambda \kappa^{1/(1-\alpha)}$, we apply this to $D = I \cap (d\Z) \cap
  S_\delta$.  Since $|D|\le |I|/d=\lambda\kappa$, 
we find from (\ref{eq:capD}) that
  $\capa(D)\leq(1+\delta\lambda)\kappa$.

  On the other hand, from Lemma~\ref{lem:mincapa} we have
  \[
  (1+\delta\lambda)\kappa \geq \capa(D) \geq c_1 \min\{|D|, (d
  |D|)^{1-\alpha}\}.
  \]
  Assume $|D| > |I|/2d = \lambda\kappa/2$, then we have either
  $1+\delta\lambda > c_1 \lambda/2$, or else $1+\delta\lambda > c_1
  (\lambda/2)^{1-\alpha}$. Neither of these hold for
  $\lambda$ large enough and $\delta$ small enough, giving a contradiction.
  Since there is nothing special about points that are $0 \pmod{d}$, we find
  that $S_\delta$ contains at most half of every congruency class in in
  $I$, and thus at most half of $I$. At this point we fix the values
  of $\lambda$ and $\delta$ to satisfy the requirement just stated, and from
  now on we treat them as constants and will hide them inside $c$-s.
  
  To complete the proof, note using the gluing formula \eqref{eq:repmu} that
  \begin{align*}
    \P\Big(\capa(A_{n+1})&-\capa(A_n) > \delta \,|\, \cF_n\Big)
    \geq \sum_{x\notin S_\delta} \P\Big(A_{n+1}=A_n\cup\{x\}\Big) \\
    & \stackrel{(\textrm{\ref{eq:repmu}})}{=} \sum_{x\notin S_\delta}
    \sum_{a\in A_n}\frac{p_{x,a}E_{A_n}(x)}{\capa A_n}
    \geq \sum_{a\in A_n} \sum_{x\notin S_\delta}
    \frac{c|x-a|^{-\alpha-1} \delta}{\kappa}
  \end{align*}
  (where in the last inequality we used that for $x\in S_\delta$ we also have $E_A(x) > \delta$.)
  For every $a\in A$ we estimate its contribution to the sum by
  applying the above with $I=[a,a+\lambda\kappa^{1/(1-\alpha)}]$. We get
\[
    \sum_{x\notin S_\delta}|x-a|^{-\alpha-1}
     \geq \sum_{x\in I \setminus S_\delta} c |I|^{-\alpha-1}
    \geq \frac{|I|}{2} \cdot c |I|^{-\alpha-1} = c
    \kappa^{-\alpha/(1-\alpha)}.
\]
Returning the term $c\delta/\kappa$ and summing over the $n$ elements of $A$ gives the claim.
\end{proof}

With this result in place, we are ready to prove our first non-trivial
bounds on the capacity of the DLA generated by an $\alpha$-walk.

\begin{lemma}\label{l:cap_lowerbound}
  Let $R$ be an $\alpha$-walk, and $(A_n)$ the associated DLA.  Then almost
  surely there exists some $c$ such that
  \[
  \capa(A_{n})>cn^{{\textstyle \frac{2-2\alpha}{2-\alpha}}}
  \]
  for all large enough $n$. Consequently, this bounds holds a.s.\ for all
  $n$ with a random $c$.
\end{lemma}

\begin{proof}
  The key is that a long as the capacity is small, the probability that it
  grows is not too small.  Let $X_n$ be the events
  \[
  X_n := \Big\{\capa(A_n) < n^{\frac{2-2\alpha}{2-\alpha}}\Big\} \cap
  \Big\{\capa(A_{n+1})-\capa(A_n) \leq \delta\Big\}
  \]
  where $\delta$ is from Lemma~\ref{lem:capabalance}.  We can bound
  $\P(X_n|\cF_n)$ by $0$ if $\capa(A_n)$ is large, and so using
  Lemma~\ref{lem:capabalance} we have
  \begin{equation}\label{eq:PXncond}
    \P(X_n\,|\,\cF_n) < 1-\frac{cn}{n^{2/(2-\alpha)}} =
    1-c_2 n^{-\frac{\alpha}{2-\alpha}}.    
  \end{equation}

  Since $\capa(A_n)$ is increasing in $n$, it suffices to prove the claim
  for $n=2^k$.  Let $B_k$ be the (bad) event that too many $X_n$-s happened
  between $2^{k}$ and $2^{k+1}$ namely
  \[
  B_k := \left\{ \#\{ X_n:2^{k}\leq n<2^{k+1}\} \geq 2^{k}\left(1-{\textstyle
        \frac{1}{2}}c_2 2^{-(k+1)\frac{\alpha}{2-\alpha}}\right)\right\}.
  \]
  By (\ref{eq:PXncond}) the $X_n$ are stochastically dominated by
  i.i.d.\ variables and hence the probability of $B_k$ may be
  estimated by standard inequalities for such sums. We get
  \[
  \P(B_{k}) \leq \exp(-c2^{k-(k+1)\frac{\alpha}{2-\alpha}})
  \leq \exp(-c2^{k\frac{2-2\alpha}{2-\alpha}}).
  \]
  Therefore $\sum \P(B_k)$ converges and a.s.\ only a finite number of
  the $B_k$ occur.  The lemma will follow once we show that for some
  $c>0$ sufficiently small $\neg B_k$ implies
  $\capa(A_{2^{k+1}}) \geq c 2^{k(2-2\alpha)/(2-\alpha)}$.

  If for any $2^k\leq n<2^{k+1}$ we have $\capa(A_n) \geq
  n^{(2-2\alpha)/(2-\alpha)}$ then 
  \[
  \capa(A_{2^{k+1}}) > \capa(A_n)
  > (2^k)^{\frac{2-2\alpha}{2-\alpha}}
  = c (2^{k+1})^{\frac{2-2\alpha}{2-\alpha}}
  \]
  Otherwise, $\neg B_k$ implies that
  \[
  \capa(A_{n+1})-\capa(A_n) > \delta \quad \text{ for at least
    $c2^{k\frac{2-2\alpha}{2-\alpha}}$ $n$-s.}
  \]
  This of course implies
  \[
  \capa(A_{2^{k+1}})\geq c\delta2^{k\frac{2-2\alpha}{2-\alpha}}
  \]
  which was to be proved.
\end{proof}

\subsection{From capacity to diameter}\label{sec:capdiam}

We will now use the various estimates on the escape probabilities and
capacity of $A_n$ proved in the previous section to get bounds on the the
diameter of $A_n$.  This section is analogous to \S 3 in part I.


\begin{thm}\label{t:diam_ge}
For an $\alpha$-walk $R$ one has
\[
\diam(A_n)\ge \left(\frac{n^2}{\capa
  A_{n}}\right)^{1/\alpha}\textrm{ infinitely often}
\]
and
\[
\diam(A_n)\ge \left(\frac{cn^2}{(\log\log n)\capa
  A_{n}}\right)^{1/\alpha}\textrm{ for all $n$ sufficiently large}
\]
both with probability 1.
\end{thm}
\begin{proof} Fix some $m$, $D$ and set $A$ with $|A|=m$, and examine the
  growth probability $\P(\diam A_{m+1}>D\,|\,A_m=A)$.  If $\diam A>D$, then
  this probability is simply 1. Otherwise, suppose that $x$ was glued to
  some $a\in A$. If $|x-a|>2D$, then $\diam A_{m+1}>D$ (in fact this even
  implies that $\diam A_{m+1} \ge 2D$) but it also implies that
  $d(x,A)>D\ge \diam A$.  Thus we can use Lemma~\ref{lem:tran_hit} to
  estimate $E_A(x) > c$.  We get
  \begin{multline}\label{eq:mDcap}
    \P\Big(\diam A_{m+1} > D\,|\,A_m=A\Big)
    \ge \sum_{\substack{a\in A,\\ x:|x-a|>2D}}\mu(x,a) =\\
    \stackrel{(\ref*{eq:repmu})}{=}
    \sum_{a\in A}\sum_{x:|x-a|>2D}\frac{E_A(x)p_{x,a}}{\capa(A)}
    \ge \sum_{a\in A}\frac{c \P(|R_1|> 2D)}{\capa(A)}
    \geq \frac{cmD^{-\alpha}}{\capa(A)}.
  \end{multline}
  Applying this for $D_m=(4m^2/\capa A_m)^{1/\alpha}$, we get that the
  events $\diam A_{m+1} > D_m$ stochastically dominate a sequence of
  independent Bernoulli trials with probability $c/m$.  By the
  Borel-Cantelli Lemma, a.s.\ infinitely often $\diam A_{m+1} > D_m >
  \left(\frac{m^2}{\capa A_m}\right)^{1/\alpha}$.

  For the second part of the theorem, we apply (\ref{eq:mDcap}) with
  \[
  D_m = \left(\frac{c_2 m^2}{(\log\log m)\capa A_m}\right)^{1/\alpha}.
  \]
  For some $c_2$ sufficiently small, we get that the events $\diam
  A_{m+1} > D_m$ dominate a sequence of independent Bernoulli trials with
  probability $(4\log\log m)/m$.  Thus
  \begin{multline*}
  \P\Big(\exists m\in[\tfrac12n,n) \text{ s.t. } \diam A_{m+1}>D_m\Big)
  \ge 1-\prod_{m=n/2}^{n-1}\left(1-\frac{4\log\log m}{m}\right)\\
  \ge 1-\left(1-\frac{4\log\log n}{n}\right)^{n/2}
  \geq 1 - \frac{c}{\log^2 n}.
  \end{multline*}
  Existence of such $m$ implies that $\diam A_n \geq D_m \geq \frac14 D_n$.
  By Borel-Cantelli, this a.s.\ holds for $n=2^k$ for all large enough $k$.
  Moving from $n=2^k$ to general $n$ only loses a constant (again by the
  monotonicity of the capacity and diameter), and the second claim of the
  theorem is proved.
\end{proof}

\begin{coro}
  Under the assumptions of Theorem~\ref{t:diam_ge},
  \begin{align*}
    \diam(A_n)&\ge c\max(n^{1/\alpha},n^2) \text{ infinitely often}\\
    \diam(A_n)&\ge \max\left(\left(\frac{n}{\log\log
          n}\right)^{1/\alpha},\frac{n^2}{\log \log n}\right) \text{ for all
      $n$ sufficiently large.}
  \end{align*}
\end{coro}

\begin{proof} The lower bound $n^{1/\alpha}$ follows from using
  Theorem~\ref{t:diam_ge} with the trivial bound
  $\capa A_n\le n$. The lower bound $n^2$ follows from applying the bound
  $\capa A_n\le C(\diam A_n)^{1-\alpha}$ (Lemma~\ref{lem:capconv}) and
  moving terms around.
\end{proof}

We now turn to our current best upper bound for the diameter for
$\alpha\in(1/3,1)$.  The following is a companion to
Theorem~\ref{t:diam_ge}, with reversed direction for the inequalities.

\begin{thm}\label{t:diam_le}
  For any symmetric random walk $R$ with $\P_0(|R_1| > x) \le C
  x^{-\alpha}$ (and in particular any $\alpha$-walk) one has a.s.
  \[
  \diam(A_n) = o\bigg(n(\log n)^{1+\eps}\max_{m<n} \frac{m}{\capa
    A_m}\bigg)^{1/\alpha}.
  \]
\end{thm}

\begin{proof}
As in the proof of Theorem \ref{t:diam_ge}, let $m$, $D$ be integers
and $A$ a set of size $m$. Condition on $A_m=A$ and assume $A_{m+1}$
was constructed by gluing some $x$ at some $a\in A$. Then clearly
$\diam A_{m+1}-\diam A_m \le |x-a|$. Hence we can write
\begin{multline}\label{eq:Deldiam}
\P(\diam A_{m+1}-\diam A_m > D \,|\,A_m=A)
  \le \sum_{\substack{a\in A\\|x-a|>D}}\mu(x,a) \\
  = \sum_{\substack{a\in A\\|x-a|>D}}\frac{E_A(x)p_{x,a}}{\capa A}
  \le \sum_{a\in A}\frac{\P_0(|R_1|>D)}{\capa A}
  \le \frac{CmD^{-\alpha}}{\capa A}.
\end{multline}
Here we do not need Lemma~\ref{lem:tran_hit} or some analog of it, because
we simply estimate $E_A(x)\le 1$.

Denote therefore $D_{m,k}=(2^{-k} nm/\capa A_m)^{1/\alpha}$ and
examine the event
\[
L_{m,k} :=  \big\{ \diam A_{m+1}-\diam A_m>D_{m,k} \big\}.
\]
By (\ref{eq:Deldiam}) we have that $\P(L_{m,k}|A_m)\le C2^k/n$. In other
words, $L_{m,k}$ (for each fixed $k$ and $m=1,\dotsc,n$) are
stochastically dominated by i.i.d.\ Bernoulli trials with probabilities
$C2^k/ n$ so we expect at most $2^k$ such events. Standard estimates for
Bernoulli trials give
\begin{equation}\label{eq:14}
\P\big(\text{at least $\lambda 2^k$ $L_{m,k}$ occurred}\big)
\le C\binom{n}{\lceil\lambda 2^k\rceil}
\left(\frac{C2^k}{n}\right)^{\lceil\lambda 2^k\rceil} \le
\left(\frac{C}{\lambda}\right)^{\lceil\lambda 2^k\rceil}
\end{equation}
which holds for all $\lambda>0$ and all $k$.

Let $\Gg_n$ be the event that for all $k$, no more than $\lambda 2^k$
of the $L_{m,k}$ occur. Notice that when $\lambda 2^k<1$ this means no
$L_{m,k}$ occur, and since $L_{m,k}$ is increasing in $k$, no
$L_{m,k'}$ occur for any $k'<k$ either. By (\ref{eq:14}), if $\lambda$
is sufficiently large then 
 $\P(\Gg_n)>1-C/\lambda$, since it suffices to check the events for
$k$ with $\lambda 2^k \ge 1/2$.

For any given $m$, if $k$ is maximal so that $\diam A_{m+1}-\diam A_m \le
D_{m,k-1}$, then $L_{m,k}$ occurs, and so 
\[
\diam A_{m+1}-\diam A_m \le \sum_k \1_{L_{m,k}}D_{m,k-1}.
\]
Assuming $\Gg_n$ we can sum over $m$ to get
\begin{align}
  \diam A_n-\diam A_{n/2}
  &=\sum_{m=n/2}^{n-1}\diam A_{m+1}-\diam A_m \nonumber \\
  & \le \sum_{m=n/2}^n\sum_k \1_{L_{m,k}}D_{m,k-1} \nonumber \\
  &\le \sum_{k:2^k\lambda\ge 1}\lambda 2^k \max_{m<n}D_{m,k-1} \label{eq:diamAnAn2}\\
  &=\lambda \left(n\max_{m<n}\frac{m}{\capa A_m}\right)^{1/\alpha}
  \sum_{k:2^k\lambda\ge 1}  2^{k(1-1/\alpha)} \nonumber\\
  &\le C\left(\lambda n\max_{m<n}\frac{m}{\capa A_m}\right)^{1/\alpha}.\nonumber
\end{align}
Now fix $\lambda=c_1(\log n)^{1+\eps}$ for some $c_1$ sufficiently
small. Applying that $\P(\Gg_n)\le C/\lambda$
for $n=2^l$ we see that
$\Gg_{2^l}$ happens for all $l>l_0$ so we get
\[
\diam A_{2^l} \le \diam A_{2^{l_0}} + \sum_{i=l_0}^l C\left(c_12^i(\log
2^i)^{1+\eps}\max_{m<2^i}\frac{m}{\capa A_m}\right)^{1/\alpha}
\]
Passing from $n=2^l$ to general $n$ only costs a constant and we get
\[
\diam A_n \le \diam A_{2^{l_0}}+C\left(c_1n(\log
n)^{1+\eps}\max_{m<n}\frac{m}{\capa A_m}\right)^{1/\alpha}.
\]
Since $c_1$ is arbitrary, the claim follows.
\end{proof}

\begin{coro}
  Let $R$ be an $\alpha$-walk. Then
  \[
  \diam A_n \le n^{2/\alpha(2-\alpha)+o(1)}
  \]
\end{coro}

\begin{proof}
By Lemma \ref{l:cap_lowerbound} we have
\[
\capa A_n > cn^{(2-2\alpha)/(2-\alpha)}
\]
for some random $c$, with probability 1. Using this in Theorem
\ref{t:diam_le} gives
\[
\diam A_n \le \left(n(\log
n)^{1+\eps}\max_{m<n}Cm^{\alpha/(2-\alpha)}\right)^{1/\alpha}=n^{2/\alpha(2-\alpha)+o(1)}.\qedhere
\]
\end{proof}
\section{Less than \texorpdfstring{$\frac{1}{3}$}{one third} moments}\label{sec:third}

In this section we will handle the case of $\alpha<\frac13$. As
already explained, we need to give an estimate for the capacity. This
is theorem \ref{T:linearcap} below --- the $\alpha<\frac13$ clause of
\thmref{all} is then an immediate corollary of it and of Theorem \ref{t:diam_le}.

\begin{thm}\label{T:linearcap} Let $R$ be a random
  walk satisfying $\P(R_1=x)\approx x^{-1-\alpha}$ for $\alpha\in (0,\frac13)$. Then the DLA
  generated by $R$ a.s.\ satisfies $\capa(A_n) = n^{1-o(1)}$.
\end{thm}

Throughout most of our analysis in this section we aim to derive properties
of the DLA at some time $n$, and to relate them to the DLA at time $\lfloor
n/\log n\rfloor $ (here and below we write $n/\log n$, implicitly taking
the integer part). Since $n$ is generally fixed, we frequently make it
implicit.

\subsection{Continuous time}

Let us introduce continuous time.  This is not strictly needed for our
analysis, but does simplify some of the proofs.  DLA in continuous time is
defined as follows. We start with $A_1=\{0\}$. Given the aggregate $A_t$,
each point $a\in A_t$ becomes active with rate $1$.  When a point $a$ is
activated, we start a random walk $(R_i)$ from $R_0=a$.  If $R$ avoids
$A_t$, then $R_1$ is added to $A_t$.  The whole step is instantaneous (the
speed of $R$ is infinite, if you like).  If $(R_i)$ hits $A_t$ then $A$
remains unchanged.  Let
\[
\tau(m) := \inf \{ t : |A_t| = m\}.
\]
It is easy to conclude from the gluing formula (\ref{eq:repmu}) that the
sequence $(A_{\tau(m)})_m$ of sets visited by the process $A_t$ is the DLA
process --- 
indeed, once one conditions on an addition of a particle at time $t$,
the probability that some $a\in A$ was the one activated is
exactly $E_{A_t}(a)/\capa(A_t)$ while the probability that the first
step was to $x$ (conditioned on the random walk from $a$ escaping to
infinity) is exactly $p_{x,a}E_{A_t}(x)/E_{A_t}(a)$, so one recovers
the gluing formula.

From now on, we shall simply refer to the
DLA in continuous time as DLA. We will also abbreviate $\tau=\tau(n)$ as
$n$ will be fixed for long parts of the analysis. Since new points are
added at rate $\capa(A_t)\in[1,|A_t|]$, it is clear that $|A_t|$ grows at
most exponentially fast and so the process is well defined for all $t$. In
light of Theorem~\ref{T:linearcap}, this is not too far from the truth.

\subsection{Split DLA}

The core of the argument is in introducing a process that allows us to
analyse the dependency between distant parts of the aggregate. We will
call this process split DLA (SDLA). This process is similar to DLA but
contains two components that grow independently. In fact SDLA is not a
single process, but a family of processes, though the dependence on the
parameters ($n$, $D$ and $q$) will be implicit.

We will construct a coupling of DLA and SDLA for all values of $q$, which
is then shown to have useful properties, while keeping the two aggregates
equal up to time $\tau(n)$ for most values of $q$ (SDLA is also constructed
in continuous time).  In fact, we believe that w.h.p.\ the two processes
are equal up to time $\tau(n)$ for all values of $q$, though proving that
is not needed for our argument.

SDLA is defined in terms of
two parameters. A carefully selected $D$ defined below, \eqref{eq:Ddef}
(depending implicitly on $n$), and an integer $q\in\N$. The SDLA consists
of {\bf two} sets $S_t$ and $\hS_t$. Initially we have $S_1 = \emptyset$
and $\hS_1=\{0\}$. The SDLA dynamics are very similar to DLA. Each of $S_t$
and $\hS_t$ evolves as an independent DLA in continuous time, with a small
exception: Each time a random walk is started at a point of $\hS$, there is
some probability that the first step $|R_1-a|$ of the random walk is large,
namely greater than $D$. On the $q$\textsuperscript{th} time that this happens, we add $R_1$
to $S_t$ (which up to that time has been empty), and we do not add it
to $\widehat{S}_t$. This is done whether or
not the random walk later hits $\hS_t$ or not. Similarly, ``the
$q$\textsuperscript{th} time'' counts initial steps greater than $D$
irrespectively of whether a point was added to $\widehat{S}$ or not.

While normally we keep the dependence on $q$ implicit, we will need it at
some points. We will use $S^q_t$ and $\hS^q_t$ to denote the two parts of
the SDLA when we wish to make this dependence explicit. The dependence on
$D$ and $n$ is always implicit. When a random walk used in the DLA begins
with a large jump, we say that a {\bf split} occurs. We denote the time of
the $q$\textsuperscript{th} split by
\[
\beta_q = \inf \{t : S^q_t \neq\emptyset\}.
\]
This is the {\bf birth time} of the $q$-SDLA. The {\bf birth point}
i.e.\ the first point in $S^q$ will be denoted by $b_q$. We will also care about the
first time at which a split occurs starting from a point $a\in S^q_t$, and
denote this by $\zeta_q$. Finally, to consider the time at which the parts
of the SDLA reaches certain sizes, we use $\tau_q(m)$ and $\htau_q(m)$. We
shall also consider the time at which the total size of the SDLA is $m$,
denoted $\sigma(m)$.

In the next few sections we will investigate SDLA as an independent
object. We will only return to the coupling of DLA and SDLA in
\S\ref{sec:coupling}. Nevertheless we might as well explain it roughly
at this point, to give the reader some perspective. The coupling is
very natural. Before the $q^\textrm{th}$ split of the DLA, we use the
same activation times for the DLA and for the $\widehat S$ part of the
SDLA, and the same random walks. They evolve the same. After the
$q^\textrm{th}$ split, as long as $A_t=\widehat{S}_t\cup S_t$ and
$\widehat{S}_t\cap S_t=\emptyset$ we again use the same activation
times and the same random walks. Now there is no deterministic guarantee that they
evolve the same, but we will see that, for $\alpha<\frac 13$ they
indeed do, with high probability. Once one of these conditions is
violated, we simply let them evolve independently (the exact condition
for ``separation'' of the DLA and SDLA is a little different, see
lemma \ref{L:SisA} below, but the above serves as a good approximation).

\subsection{Splitting estimates}

We now choose the parameter $D$ for the SDLA. Define
\[
M_n :=
\Med \left(\sum_{i=1}^{n}\frac{1}{\capa(A_{\tau(i)})} ; \frac{5}{6} \right)
\qquad M := M_{n/\log n},
\]
where $\Med(X;p)$ is the ``$p$-Median'' (quantile) of the variable $X$
i.e.
\begin{equation}\label{eq:Med}
\Med(X;p) = \sup\{t: \P(X<t) < p \}.
\end{equation}
Note that points are added to $A$ at rate $\capa(A)$, so
$\sum_{i=1}^{n}\frac{1}{\capa(A_{\tau(i)})}$ is an estimate for $\tau(n)$.
With this $M$, define
\begin{equation}\label{eq:Ddef}
  D := \left(\frac{6C_1 n M}{\log n} \right)^{1/\alpha},
\end{equation}
where $C_1$ is such that the random walk jump distribution satisfies
$\P(|X|>t) < C_1 t^{-\alpha}$ for all $t$.

Note that $\capa(A) \leq |A|$, and therefore $M \geq \sum_{1}^{n/\log n}
i^{-1} \geq c \log n$. It follows that
\begin{equation}\label{eq:Dgeq}
D \geq c n^{1/\alpha}.
\end{equation}
We show
below that this is not far from the truth: $M = n^{o(1)}$. The purpose of
this definition of $D$ is to have both lower and upper bounds on the
occurrence of such large jumps and in this way to control the branching of
the SDLA.

\begin{lemma}\label{L:few_splits}
  For $n$ sufficiently large, the probability that a DLA splits
  before (continuous) time $\min(2M,\tau(n/\log n))$ is at most $\frac13$.
\end{lemma}

\begin{lemma}\label{L:many_splits}
  There exists a $c_2>0$ so that
\[
\P\Big(\textrm{$A$ splits fewer than $\log n$ times by $\tau(n/2)$ and
  $\capa(A_\tau) < c_2 n/M$} \Big)
\le C n^{-c}.
\]
\end{lemma}

Since the splits of DLA and SDLA are the same, up to the
$q^\textrm{th}$ split, these lemmas give information also on SDLA. We
will use them mainly for DLA, though.

We will see that $\tau(n/\log n)$ is typically less than $2M$. Thus we
argue that $A$ is likely to accumulate $n/\log n$ points without splitting,
but very likely to split many times before accumulating $n/2$ points,
unless it has large capacity.

\begin{proof}[Proof of Lemma~\ref{L:few_splits}]
  Given $A_t$, the rate at which splits occur is the probability of a large
  random walk jump times $|A_t|$, namely $|A_t|\P(|X|>D) \leq |A_t|\cdot
  C_1 D^{-\alpha}$. Using the definition of $D$ \eqref{eq:Ddef}, as long as
  $|A| < n/\log n$ the rate of splits is at most $\frac{1}{6M}$. Therefore
  the number of splits up to time $\min(2M,\tau(n/\log n))$ is
  stochastically dominated by a Poisson process with rate $\frac{1}{6M}$
  and time interval $2M$. It follows that the expected number of splits by
  time $\min(2M,\tau(n/\log n))$ is at most $\frac13$.
\end{proof}

\begin{proof}[Proof of Lemma~\ref{L:many_splits}]
  The
  probability of a large random walk jump is at least $c D^{-\alpha}$, so
  the splitting rate is at least
  \[
  |A_t| c D^{-\alpha} = \frac{c (\log n)|A_t|}{n M}.
  \]
  The rate at which new points are added to $A_t$ is $\capa(A_t)$. Thus the probability that a split occurs before a point is added to $A_t$
  is at least
\[
\frac{c (\log n)| A_t|}{n M \capa(A_t)}.
\]
With this $c$,
  fix $c_2 = c/20$. If $|A_t|\geq n/4$ and $\capa(A_t) < c_2 n/M$
  then the probability of a new split is at least $5(\log n)/n$.

  Let $X_k$ be the indicator of the event that a split occurs in the
  interval $(\tau(k),\tau(k+1)]$, or $\capa(A_{\tau(k)}) \geq c_2
    n/M$. Then for $k\in[\frac{1}{4}n,\frac{1}{2}n]$, the $X_k$'s
  stochastically dominate i.i.d.\ Bernoulli random variables with mean
  $5(\log n)/n$, and so
  \[
  \P\bigg(\sum_{n/4}^{n/2} X_k < \log n\bigg) < C e^{-c\log n} = C n^{-c}.
  \]
  However, if $\sum_{n/4}^{n/2} X_k \geq \log n$ then either $\log n$
  splits occur or else at some time $\capa(A_t)$ exceeded $c_2
    n/M$. In the latter case monotonicity implies
  $\capa(A_\tau)>c_2 n/M$.
\end{proof}

\subsection{Growth estimates}

We shall need the following estimates for the growth rate of the DLA and SDLA.

\begin{lemma}\label{L:tonlogn}
  $\P\left(\tau(n/\log n) > 2M \text{ and }
    \sum_{k=1}^{n/\log n} \frac{1}{\capa(A_{\tau(k)})} \leq M \right)
  < C n^{-c}$.
\end{lemma}

Since by definition $\sum \frac{1}{\capa(A_{\tau(k)})}$ is quite likely to
be less than $M$, it follows that $\tau(n/\log n)$ is usually less than
$2M$. Before we start with the proof of lemma \ref{L:tonlogn} let us
note the following

\begin{lemma}\label{L:condpois}
Let $X_1,\dotsc,X_N$ be independent Poisson clocks (not necessarily
with the same rates). Let $T$ be the first time one of them activated,
and let $E$ be the event that the first to activate was $X_1$. Then
$T$ conditioned on $E$ has the same distribution as $T$.
\end{lemma}
\begin{proof}This is a straightforward calculation and we omit the
  details.
\end{proof}

\begin{proof}[Proof of lemma~\ref{L:tonlogn}]
  Points are added to $A_t$ at rate $\capa(A_t)$. Let $X_m =
  \tau(m+1)-\tau(m)$ be the time it takes to add the $m+1$\textsuperscript{st} point. Let
  $Q_m = \capa(A_{\tau(m)})$ be the rate at which it is added, and $\F$ the
  sigma-field generated by all $Q_m$'s. Examine one $X_m$ conditioned
  on $\F$. The conditioning gives us the capacity of $A_m$ as well as
  that of $A_{m+1}$ which gives some information on the point that was
  activated to increase $A_m$. Nevertheless, by lemma \ref{L:condpois}
  this information is irrelevant and we get that
  $X_m$ conditioned on $\F$ is an exponential random variable with
  mean $1/Q_m$. Clearly, conditioning on $\F$ makes the different $X_m$-s
  independent. Thus their sum can be analyzed by standard techniques as follows.

  We have that
  \[
  \E \left(e^{X_m/2}\,|\,\F\right) = \frac{Q_m}{Q_m-1/2} =1+\frac{1}{2Q_m-1}\leq e^{1/(2Q_m-1)}.
  \]
  Since $Q_m \geq 1$, it follows that
  \[
  \E \left( e^{(X_m-1/Q_m)/2}\,|\,\F \right) \leq e^{1/2Q_m^2}.
  \]
  By Lemma \ref{lem:mincapa} we have $Q_m = \capa(A_{\tau_m}) \geq c
  m^{1-\alpha}$, so $\sum_m\frac{1}{2Q_m^2} < C$, and so
  \[
  \E\exp \bigg(\tfrac12 \sum_{m<n/\log n} X_m - \tfrac{1}{Q_m}\bigg) \leq C
  \]
  (note that we integrated $\F$ away). Since $\tau(n/\log n) = \sum_{m<n/\log n} X_m$ we find
  \[
  \P\left(\tau(n/\log n) > M + \sum \frac{1}{Q_m}\right) \leq C e^{-M/2} <
  C n^{-c}
  \]
  (using $M>c\log n$). The claim follows.
\end{proof}

\begin{lemma}\label{L:24M}
With probability at least $\frac{1}{2}-Cn^{-c}$ one has that
\[
\tau(n/\log n) \le \min\{2M, \beta_1 \}
\]
where $\beta_1$ is the time of the first split.
\end{lemma}
\begin{proof}
By lemma \ref{L:tonlogn} and the definition of $M$,
\[
\P(\tau(n/\log n)\le 2M)\ge \frac 56 - Cn^{-c}.
\]
By lemma \ref{L:few_splits}
\[
\P(\beta_1\ge\min\{2M,\tau(n/\log n)\})\ge \frac 23.
\]
When both events happen (which happens with probability $\ge \frac 12
-Cn^{-c}$) we get the required inequality.
\end{proof}
Recall that for some fixed $q$, $\sigma(m)$ is the (continuous) time that the size of
the SDLA (namely $|S_t|+|\hS_t|$) first reaches $m$.

\begin{lemma}\label{L:halfnton}
  Fix $q\in\N$, and consider the event $\cB$ that $\sigma(n) - \sigma(n/2)
  < 2M$ and $\capa(S_{\sigma(n)}) + \capa(\hS_{\sigma(n)}) < n/5M$. Then
  $\P(\cB) < C e^{-cn}$.
\end{lemma}

\begin{proof}
  Note that $\capa(S_t)$ (resp.\ $\capa(\hS_t)$) is the rate at which
  points are added to $S_t$ (resp.\ $\hS_t$).  Consider the process $N_t$
  that counts points being added to $S_t$ or $\hS_t$ starting at time
  $\sigma(n/2)$, stopped at the first time $t$ when $\capa(S_t) +
  \capa(\hS_t) > \frac{n}{5M}$.  Then $N_t$ is stochastically dominated by a
  Poisson process with intensity $\frac{n}{5M}$.  Up to this stopping time,
  we have $N_t = |S_{\sigma(n/2)+t}| + |\hS_{\sigma(n/2)+t}| - n/2$.

  Thus the number of points added within time $2M$ and before the stopping
  time is dominated by a $\textrm{Poi}(2M\cdot\frac{n}{5M})$ variable. The
  probability that this exceeds $n/2$ is exponentially small. 
  However, if $N_{2M} < n/2$ then either $\sigma(n)-\sigma(n/2) >
  2M$ or else $\capa(S_t) + \capa(\hS_t)$ exceeds $\frac{n}{5M}$ before
  time $\sigma(n)$.  The claim follows by monotonicity of capacity.
\end{proof}

\subsection{Interaction probabilities}\label{sub:interaction}

We now analyze the probability of an interaction between $S^q$ and $\hS^q$.
The bounds we get will imply that for any given $q$, the law of the DLA
$A_t$ is close to the law of $S_t \cup \hS_t$.
Fix some $q\in\N$, and consider the SDLA. The set $S_t\cup
\hS_t$ evolves very similarly to a DLA. The difference is that if a random
walk from $S_t$ hits $\hS_t$ or vice versa, then a point may be added to
the SDLA but not to the DLA. Thus we need to bound the probability of such
an intersection happening.

We shall be interested in the union of all random walk trajectories that
originate from a point of $S^q$ or $\hS^q$ up to time $t$.  More precisely,
$T^q_t$ is the union of all trajectories of random walks that started from
points of $S^q$ up to time $t$ and escaped (i.e.\ led to the addition of a
point to $S^q$).  Similarly, $\hT^q_t$ is the union of all trajectories of
random walks that started from points of $\hS^q$ up to time $t$ and
escaped.  However, the walk at time $\beta_q$ is treated differently.  Even
though it starts at a point of $\hS^q$, its first step is to $b_q\in S^q$,
and we include this walk minus its starting point in $T^q$ and not in
$\hT^q$.  Note that paths in both $T^q$ and $\hT^q$ include their starting
points, so $S^q\subset T^q$ and $\hS^q\subset\hT^q$.

Our goal in this section is the following:

\begin{lemma}\label{L:TSnu}
  For any $q\in\N$,
  \[
  \P\left(T^q_\sigma \cap \hS^q_\sigma \neq\emptyset\right) \leq C n^{-c}
  \qquad \text{and} \qquad
  \P\left(\hT^q_\sigma \cap S^q_\sigma \neq\emptyset\right) \leq C n^{-c}.
  \]
\end{lemma}

We first argue that after a large jump the walk is unlikely to hit any
given point.

\begin{lemma}\label{L:jumpreturn}
  Consider an $\alpha$-random walk $R$ started at $a$, and a point $z$, and let $R_1$
  be the first jump of the random walk. Then $\P_a(R\text{ hits }z,
  |R_1-a|>L) \leq C/L$. 

Consequently, for any fixed set $S$, the
  probability of making a jump of size at least $L$ and hitting $S$ is at
  most $C|S|/L$.
\end{lemma}

\begin{proof}
  With $R_1=a+X$, we use that Green's function satisfies $G(z-R_1)
  \approx |z-R_1|^{\alpha-1}$ (lemma \ref{lem:G}) to get
  \[
  \P_a(|R_1-a|>L, R\text{ hits }z) \leq C \sum_{|X|>L} |X|^{-\alpha-1}
  |z-a-X|^{\alpha-1}.
  \]
  In this last sum, terms with $|z-a-X|\leq L$ are bounded by
  $\sum_{-L}^{L} L^{-\alpha-1} |i|^{\alpha-1} \leq C L^{-1}$. Terms with
  $|z-a-X| > L$ are bounded by $\sum_{|X|>L} |X|^{-\alpha-1} L^{\alpha-1}
  \leq C/L$ as well.
\end{proof}

Recall that $b_q$ is the birth point, the first point in $S^q$. By $A-x$ we
denote the translation of a set $A$ by $x$.

\begin{lemma}\label{L:Tsmall}
  Let $T$ be either $T^q_{\tau^q(n)}-b_q$ or $\hT^q_{\htau^q(n)}$ for some
  $q\in\N$, and let $I$ be any interval of length $L\geq D$. Then
  \[
  \E |T\cap I| \leq C n L^\alpha.
  \]
\end{lemma}

Roughly, this is so since we are considering $n$ random walks, and each of
those visits $L^\alpha$ points in $I$. This is somewhat complicated by
the fact that there are also random walks that do not escape $S^q$ but
still visit $I$.

\begin{proof}
  We may assume (by increasing $C$ if necessary) that $L$ is
  sufficiently large. We consider only the case of $T^q$, the case of $\hT^q$ is proved identically.
  We show the stronger fact, that over the time it takes to add a single
  point to $S^q$ the expected number of points visited in $I$ is at most
  $C L^\alpha$.

  Examine the random walk at some time when it is in $I$. The random
  walks we are dealing with have probability at least $cL^{-\alpha}$
  to make a step bigger than $2L$. By lemma \ref{lem:tran_hit} once this happens it
  has probability bigger than some $c$ to never hit $I$ again, hence
  it has probability at least $cL^{-\alpha}$ to make a large step and
  then never hit $I$ again.

  Next apply the previous lemma. We get that (again, after the initial
  step bigger than $2L$) the probability that the random walk hits
  $S^q$ is at most $C|S^q|/L$. Since $|S^q| \le n$ and since $L \ge D
  \ge cn^{1/\alpha}$ (by (\ref{eq:Dgeq})) we get that $C|S^q|/L \le
  CL^{\alpha-1} \ll L^{-\alpha}$ (here we only need $\alpha
  <\frac{1}{2}$). Hence we get, for $L$ sufficiently large, that with probability at least
  $cL^{-\alpha}$ the random walk makes a step bigger than $2L$ and
  after that disappears to infinity, never returning to either $I$ or
  $S^q$. This, of course, adds a point to $S^q$.

  Since all these calculations were independent of the past, we get
  that the number of points visited in $I$ before adding a single
  point to $S^q$ is stochastically dominated
  by a geometric random variable with expectation $CL^\alpha$, proving
  the lemma.
\end{proof}


\begin{lemma}\label{L:Tmissx}
  Let $T$ be either $T^q_{\tau^q}-b_q$ or $\hT^q_{\htau^q}$ for some
  $q\in\N$. Fix any $x\in\Z$, and let $\nu$ be an independent random walk
  step conditioned on $|\nu|\geq D$. Then
  \[
  \P(x-\nu\in T) \leq C n D^{\alpha-1} \leq C n^{2-1/\alpha}.
  \]
\end{lemma}

\begin{proof}
  Let $I_k = [2^k D, 2^{k+1}D)$. Then we can write
  \[
  \P(x-\nu\in T) = \sum_{k\geq 0} \P(\nu\in I_k, x-\nu\in T) +
  \sum_{k\geq 0} \P(-\nu\in I_k, x-\nu\in T).
  \]
  Since $\P(\nu=y) \leq C D^\alpha y^{-\alpha-1}$ (the $D^\alpha$ comes
  from the conditioning of $\nu$ to be large), we have for $y\in I_k$
  that $\P(\nu=y) \leq C 2^{-(\alpha+1)k} D^{-1}$. Thus
  \begin{align*}
    \P(\nu\in I_k, x-\nu\in T) &= \sum_{y\in I_k} \P(\nu=y) \P(x-y\in T) \\
    &\leq C 2^{-(\alpha+1)k} D^{-1} \sum_{y\in I_k} \P(x-y\in T) \\
    &= C 2^{-(\alpha+1)k} D^{-1} \E |T\cap(x-I_k)| \\
    \textrm{By Lemma \ref{L:Tsmall}}\qquad
      &\leq C 2^{-(\alpha+1)k} D^{-1} n |I_k|^\alpha = C 2^{-k}
    nD^{\alpha-1}.
  \end{align*}
  The same bound holds for $-\nu\in I_k$.  Summing over $k$ now gives
  $\P(x-\nu\in T) \leq C n D^{\alpha-1}$.  The last claim holds since $D >
  c n^{1/\alpha}$ (\ref{eq:Dgeq}).
\end{proof}
We remark that this last lemma is the most important point in the proof
where we need a pointwise estimate on $\P(\nu=y)$ and cannot do with
an estimate on $\P(\nu>y)$.

\begin{proof}[Proof of Lemma~\ref{L:TSnu}]
  We prove only the first bound, as the second is proved identically.

  Let $\nu$ be the jump of the walk the creates $S^q$ i.e.~the
  $q$\textsuperscript{th} large jump. Let $a$ be the point of
  $\hS^q$ from which the jump occurred (so $b_q = a+\nu$). The first key observation
  is that $\nu$ is independent of $\hS^q$, and only affects $S^q$ by a
  translation, so it is independent also of $S^q-\nu$.

  We have that $S^q_t-a-\nu$ and $\hS^q_t$ are also independent for any
  $t$. Note that for the stopping time $\sigma$, the independence fails,
  since their sizes are now linked. To overcome this, we consider
  intersections among the larger sets $\hS^q_{\htau^q}$ and $T^q_{\tau^q}$.
  For those we have that $\hS^q_{\htau^q}$, $T^q_{\tau^q}-a-\nu$ and $\nu$
  are jointly independent, with $\hS^q_{\htau^q}$ being some set of size
  $n$ and $\nu$ being a random walk step, conditioned to be large.
  Moreover, $a$ depends only on $\hS^q_{\htau^q}$.

  Denote $S=\hS^q_{\htau^q}$ and
  $T=T^q_{\tau^q}-a-\nu$, so we are interested in the probability that
  $S-a-\nu$ intersects $T$. Condition on $a,S$, and consider any $s\in S$.
  By Lemma~\ref{L:Tmissx} we have that $\P(s-a-\nu\in T) \leq C
  n^{2-1/\alpha}$. Since $|S|=n$ it follows that $\P(S-a-\nu \cap T \neq
  \emptyset) \leq C n^{3-1/\alpha}$. Since this bound is uniform in
  $S$ and $a$, and since $\alpha<1/3$, this completes the proof.
\end{proof}

The last part of this proof is the only place where $\alpha<\frac{1}{3}$ is
crucial to our proof. It is also used in the next lemma, though a weaker
statement which should hold for $\alpha<\frac{1}{2}$ would suffice there.
However, Lemma~\ref{L:TSnu} fails for $\alpha\in[\frac{1}{3},\frac{1}{2}]$,
as we expect there to be intersections.  Using a weaker form of this lemma
will probably require significant modification of our argument.

\subsection{Bounds on capacity influence}

Let us extend the definition of the Green's function to sets $A,B\subset
\Z$ by
\[
G(A,B) = \sum_{\substack{a\in A \\ b\in B}} G(a,b).
\]
The reason for this is the easy bound $\capa(A\cup B) \geq \capa(A) +
\capa(B) - G(A,B)$, which holds since the capacity is the sum of the
escape probabilities. With this in mind we prove the following.

\begin{lemma}\label{L:Gnu}
  For any $q\in\N$ we have $\E G(S^q_\tau,\hS^q_\tau) \leq C$.
\end{lemma}

\begin{proof}
  Since $\tau_q,\htau_q \geq \tau$ it is enough to prove that $\E
  G(S^q_{\tau_q},\hS^q_{\htau_q}) \leq C$, i.e.\ allow each of the two sets
  to continue to grow to size $n$.

  As before, let $\nu$ be the size of the $q$\textsuperscript{th} large jump that gave rise
  to $S^q$. Denote $A=\hS^q_{\htau_q}$ and $B = S^q_{\tau_q} - \nu$, and
  note that $A,B$ are sets of size $n$ and $\nu$ a random walk step
  conditioned to be large, independent of $A,B$ (the independence
  claim is as in the previous lemma). Conditioning on $A,B$ we
  get
  \[
  \E G(S^q_{\tau^q},\hS^q_{\htau^q})
  = \sum_{\substack{x\in A\\y\in B}} \E G(x-(y+\nu))
  \]
  where the expectations are only over $\nu$. Recall that
  $G(x)\approx|x|^{\alpha-1}$ by lemma \ref{lem:G} and that
  $\P(\nu=z)< C/D$ for all $z$. Therefore for any
  $x$ we have
\begin{align*}
\E G(x-\nu)&\le \sum_{y:|x-y|\le D}C|x-y|^{\alpha-1}\cdot
\P(\nu=y)+CD^{\alpha-1}\P(|\nu-x|>D)\\
&\le \frac{C}{D}\sum_{i=1}^Di^{\alpha-1}+CD^{\alpha-1}\le
C D^{\alpha-1} \stackrel{\textrm{(\ref{eq:Dgeq})}}{\leq} C n^{1-1/\alpha}.
\end{align*}
Summing over the sets $A$ and $B$ gives
  \[
  \E G(S^q_{\tau_q},\hS^q_{\htau_q}) \leq C n^{3-1/\alpha} \leq C. \qedhere
  \]
\end{proof}

\subsection{Coupling DLA and SDLA}\label{sec:coupling}

We now have in place all the necessary estimates about DLA and SDLA. In
order to put them to use we need to describe a coupling between the
processes. More precisely, we construct the DLA as well as all $q$-SDLAs (for
all $q\in\N$) in the same probability space. This coupling will satisfy the
following two properties, which we formulate as lemmas:

\begin{lemma}\label{L:SisA}
  If for some $q$ and $t$ we have $T^q_t \cap \hS^q_t = \hT^q_t \cap S^q_t
  = \emptyset$, then $A_t=S^q_t\cup \hS^q_t$ (as a disjoint union).
\end{lemma}

Recall that $T^q$ is the union of all walks used to create $S^q$ as
defined in the beginning of \S\ref{sub:interaction}. Recall also that $\beta_q$ is the time the $q$\textsuperscript{th} split
occurs (i.e.\ that $S^q$ is created), that $b_q$ is the first point in
$S^q$, and and that $\zeta_q$ is the first
time a split occurs from a point of $S^q$.

\begin{lemma}\label{L:iid}
  The processes $\big(\{S^q_{t+\beta_q}\} - b_q:t\leq
    \zeta_q-\beta_q\big)$, i.e.\ $S^q$ killed at its first split, form
  an i.i.d.\ sequence.  Each has the law of DLA killed once a random
  walk begins with a large jump.
\end{lemma}

Thus we have on the one hand that the SDLAs for different $q$ are
independent until they first split, and on the other hand they are all part
of the large DLA process to the extent possible.

The coupling is constructed as follows. Start with the DLA process $A_t$.
For any $q\in\N$, the $q$-SDLA has $\hS^q_t=A_t$ until the $q$\textsuperscript{th} time a
random walk starts with a large jump, at which time $S^q_t$ becomes
non-empty. As long as $A_t=S^q_t\cup\hS^q_t$, there is a natural way to
continue the coupling: the same points are activated in both processes and
the same random walks used from active points. As long as $T^q_t \cap
\hS^q_t = \hT^q_t \cap S^q_t = \emptyset$, points are added to $A_t$ if and
only if they are also added to one of $S^q_t$ or $\hS^q_t$.

This guarantees that Lemma~\ref{L:SisA} holds, regardless of how $S^q$ and
$\hS^q$ evolve once there is some interaction between them. In order to
achieve the independence property of Lemma~\ref{L:iid}, we say that once
$T^q_t \cap \hS^q_t \neq \emptyset$ or $\hT^q_t \cap S^q_t \neq \emptyset$,
the $q$-SDLA continues its evolution independently of the DLA and all other
$q'$-SDLAs. Lemma~\ref{L:iid} now holds since the point activation and
associated random walks used to generate $S^q$ over the time interval
$[\beta_q,\zeta_q]$ are all independent and disjoint of those used for any
other $q$.

To convince oneself that the activations and random walks are really
disjoint, colour $A_t$ as follows: the first point is coloured 0, and
whenever a point $a$ is activated and the random walk adds a new point $b$
to the DLA, $b$ inherits $a$'s colour, except when a split occurs,
in which case, if this is the $q$\textsuperscript{th} split, $b$ gets
coloured by $q$ (the ``colours'' are elements of $\{0,1,2,\dotsc\}$). It is
easy to check that $S^q$ is exactly the points that have colour $q$ until
one of the two events happen:
\begin{itemize}
\item either $T^q \cap \hS^q \neq \emptyset$ happened, which is one of
  the two ``bad'' events of lemma \ref{L:SisA}, so after that
$S^q$ becomes independent of everything else,
\item or $\zeta_q$ occurs, in which case a split happens from $S^q$
  and $S^q$ now corresponds to two colours in $A$.
\end{itemize}
This shows that the portions of $S^q$ during $[\beta_q,\zeta_q]$ are
independent, since their randomness either comes from differently coloured
points, or from a completely independent source.\qed

Now that lemmas \ref{L:SisA} and \ref{L:iid} are proved, we have all
the necessary pieces.

\subsection{The grand assembly}

We are now ready to put together all the parts of the argument, and prove
Theorem~\ref{T:linearcap}. The next lemma is the core of the proof. It uses
everything we learned so far, and the theorem follows from it by a simple
induction.  Recall that $\Med(\cdot;\frac16)$ is the $\frac16$-quantile
defined in \eqref{eq:Med}.

\begin{lemma}\label{L:highcap}
  For some $c_3$, with probability at least $1 - C n^{-c}$,
  \[
  \capa(A_n) \geq \min\left\{ \frac{c_3n}{M},\frac{\log n}{5} \Med \big(
    \capa(A_{n/\log n}); \tfrac16 \big) \right\} .
  \]
\end{lemma}

\begin{proof} We may assume $n$ is sufficiently large, as for small
  $n$ the lemma is true trivially if the $C$ in the probability
  estimate $1-Cn^{-c}$ is made large enough. For every $q\in\N$ Let $S^q$
  be a $q$-SDLA process, and assume that they are all coupled to our DLA as
  in the previous section.
  With high probability the (bad) event of Lemma~\ref{L:many_splits} does not
  occur for the DLA and the (bad) events of Lemmas~\ref{L:halfnton} and
  \ref{L:TSnu} do not occur for any $q\leq\log n$. Here and below
  ``with high probability'' means with probability at least
  $1-Cn^{-c}$. To spare the reader
  some page flipping here is a short reminder of these lemmas (up to the
  aforementioned bad events):
\begin{list}{}{}
\item[Lemma \ref{L:many_splits}] Either the DLA splits $\log n$
  times or it has large capacity.
\item[Lemma \ref{L:halfnton}] Either the SDLA satisfies
  $\sigma(n)-\sigma(n/2)\ge 2M$ or it has large capacity.
\item[Lemma \ref{L:TSnu}] $T\cap \hS = \hT \cap S = \emptyset$.
\end{list}
  Thus we suppose from
  here on that this is the case. Recall that $\tau$ is the stopping
  time when the DLA reaches size $n$ and $\sigma$ is the analogous
  quantity for the $q$-SDLA.  From Lemma \ref{L:SisA} we now see that
  $\sigma(m) = \tau(m)$ for any $m\le n$ and for any $q\leq \log n$.

  If for some $q\leq \log n$ we have $\capa(S_\sigma) + \capa(\hS_\sigma)
  \geq n/5M$ then at least one of them has capacity at least
  $n/10M$. Since $A_\tau = S_\sigma \cup\hS_\sigma$ (again using Lemma
  \ref{L:SisA}) by the monotonicity of capacity we get
  $\capa(A_\tau)\geq n/10M$ and we are done as long as
  $c_3\leq\frac{1}{10}$. Thus suppose this too is not the case, and
  so, since the bad event of Lemma~\ref{L:halfnton} did not occur, we get
  $\tau(n)-\tau(n/2) \ge 2M$. By making $c_3<c_2$ (with $c_2$ taken from
  Lemma~\ref{L:many_splits}) we may similarly
  assume (this time using Lemma~\ref{L:many_splits})
  that $A$ splits at least $\log n$ times by time $\tau(n/2)$.

  Call a split $S^q$ {\bf good} if the following both hold:
  \begin{enumerate}
  \item \label{enu:typcap} its capacity is typical: 
    $\capa(S^q_{\tau_q(n/\log n)}) \geq
     \Med\big( \capa(A_{\tau(n/\log n)}); \frac16 \big)$, and
  \item \label{enu:fastgrow} it grows fast enough: $\tau_q(n/\log n) <
    \min\{\beta_q + 2M,\zeta_q\}$, that is, $S^q$ reaches size $n/\log n$
    both before splitting ($\zeta_q$) and within $2M$ of becoming non-empty
    ($\beta_q$).
  \end{enumerate}
  Since this only depends on the process $S^q$ until it splits, the events
  that $S^q$ are good are i.i.d.\ (by Lemma~\ref{L:iid}).  By the
  definition of $\Med$, the probability of clause \ref{enu:typcap} holding
  for any $q$ is at least $\frac{5}{6}$.  On the other hand,
  Lemma~\ref{L:24M} gives that clause \ref{enu:fastgrow} holds with
  probability at least $\frac{1}{2}-Cn^{-c}$. Hence
  each $S^q$ is good with probability at least $\frac13 - C n^{-c}$.
  Thus (for $n$ large enough) the probability that at least
  $\frac14\log n$ of the $S^q$ for $q\leq\log n$ are good is at least
  $1-Ce^{-c\log n} = 1-Cn^{-c}$. Assume this is the case.

  To summarize our current assumptions, either $\capa(A_\tau)>\frac{c_3
    n}{M}$, or else there are at least $\frac{1}{4}\log n$ good branches
  $S^q$ with
  $q<\log n$, each of which started before time $\tau(n/2)$ and accumulated
  $n/\log n$ points without splitting and within time $2M$. Since also
  $\tau(n)-\tau(n/2) \geq 2M$, these points are all present in $A_\tau$ and
  separate from those of other good branches. Moreover, the first $n/\log
  n$ points in each of these branches have a typically large capacity.

  Let $\Gg = \{q<\log n, S^q \text{ is good}\}$, so that $A_\tau$
  contains $\bigcup_{q\in\Gg} S^q_{\tau_q(n/\log n)}$.  This implies the
  capacity bound
  \[
  \capa(A_\tau) \geq \sum_{q\in\Gg} \capa(S^q_{\tau_q(n/\log n)}) - C
  \sum_{q\in\Gg} G(S^q_\tau, \hS^q_\tau).
  \]
  (In fact, it suffices to subtract the Green's function between
  $S^q_{\tau_q(n/\log n)}$ and the union of that set for other good $q$'s,
  which is smaller.)  Now, for each good $S^q$ we have
  $\capa(S^q_{\tau_q(n/\log n)}) \geq \Med\big( \capa(A_{\tau(n/\log n)});
  \frac16 \big)$, and each term in the second sum has bounded expectation
  (by Lemma~\ref{L:Gnu}). It follows that
  \[
  \E \sum_{q\in\Gg} G(S^q_\tau, \hS^q_\tau) \leq 
  \E \sum_{q=1}^{\log n} G(S^q_\tau, \hS^q_\tau) \leq C \log n
  \]
  and so this sum exceeds $n^c$ only with probability $C n^{-c} \log n$. If
  it does not then because $\capa(A_{\tau(n/\log n)}) \geq c(n/\log
  n)^{1-\alpha}$ (Lemma \ref{lem:mincapa}),
  \begin{align*}
    \capa(A_\tau) &
    \geq \frac{\log n}{4} \Med\big( \capa( A_{\tau(n/\log n)}); \tfrac16
    \big) - Cn^{c} \\
    & \geq \frac{\log n}{5} \Med\big( \capa(A_{\tau(n/\log n)}); \tfrac16
    \big),
  \end{align*}
  (for $c$ small and $n$ large). This completes the proof.
\end{proof}

\begin{lemma}\label{L:linearcap}
  $M_n = n^{o(1)}$, and $\Med(\capa A_\tau(n);\tfrac{1}{6}) = n^{1-o(1)}$.
\end{lemma}

\begin{proof}
  We now no longer need the SDLA process, so to keep notations clear,
  denote $Q_n = \capa(A_{\tau(n)})$.  All we need about these random
  variables is that $Q_1=1$, $Q_{n+1}\ge Q_n$, that $M_n$ is (by
  definition) the $\frac56$-median of $\sum^{n} Q_i^{-1}$, and that
  with high probability
  \begin{equation}\label{eq:Qbound}
    Q_n \geq \min\left\{ \frac{c_3 n}{M_{n/\log n}},\frac{\log n}{5} \Med \big(
      Q_{n/\log n}; \tfrac16 \big) \right\} .
  \end{equation}
  In particular, for some $\ell_0$, \eqref{eq:Qbound} holds simultaneously for
  all $n=2^\ell$, $\ell\geq \ell_0$ with probability at least
  $\frac56$. Call the event that this happens $\Gg$.

  Fix $\eps>0$, and make $\ell_0$ larger if needed, so that $\big(\log
  2^{\ell_0}\big)^\eps > \max\{2 / c_3\eps, 10\}$.  (This can only increase
  $\Gg$.)  We now pick some $K=K(\ell_0)$ sufficiently large such that for
  all $n\leq 2^{\ell_0}$
  \begin{equation}\label{eq:MQinfuct}
    M_n\leq K n^{\eps}, \qquad \text{ and } \qquad
    Q_n \geq \frac{n^{1-\eps}}{K\eps}.
  \end{equation}
  We now prove by induction that on the event $\Gg$, \eqref{eq:MQinfuct}
  holds for all $n$ (the left clause in \eqref{eq:MQinfuct} is just an
  inequality of numbers so we just show that it holds --- the right
  clause is an event, and we show that it follows from $\Gg$).

  To see this, consider some $n>2^{\ell_0}$, and let $n'=2^\ell\leq n$ be
  the largest power of $2$ below $n$.  We first show the right clause in
  \eqref{eq:MQinfuct}.  Since $Q_n$ is monotone, applying \eqref{eq:Qbound}
  with $n'$ (which we are allowed, because we assume the event $\Gg$
  happened) we get either
  \begin{equation}\label{eq:cases}
  Q_n \geq \frac{c_3 n'}{M_{n'/\log n'}}
  \qquad \text{ or } \qquad
  Q_n \geq \frac{\log n'}{5} \Med \big(Q_{n'/\log n'}; \tfrac16 \big).
  \end{equation}
  Under the induction hypothesis \eqref{eq:MQinfuct} for $n'/\log n'$, the
  former case implies
  \[
  Q_n \geq \frac{c_3 n'}{M_{n'/\log n'}}
  \geq \frac{c_3 n'}{K (n'/\log n')^\eps}
  \geq \frac{c_3}{2K} n^{1-\eps} (\log n')^\eps.
  \]
  By our assumption that $\big(\log 2^{\ell_0}\big)^\eps > 2/c_3\eps$, this
  case yields $Q_n > n^{1-\eps}/(K \eps)$, as needed.  For the latter case
  in \eqref{eq:cases} note that because 
  $\P(\Gg)>\frac{5}{6}$, the induction hypothesis implies that
  $\Med(Q_m;\frac{1}{6})\ge m^{1-\eps}/(K\eps)$ for all $m<n$ and in
  particular for $n'/\log n'$ so we get
  \[
  Q_n \geq \frac{\log n'}{5} \Med \big(Q_{n'/\log n'}; \tfrac16 \big)
  \geq \frac{\log n'}{5} \frac{(n'/\log n')^{1-\eps}}{K\eps}
  \geq \frac{1}{10 K\eps} n^{1-\eps} (\log n')^\eps.
  \]
  By our assumption that $\big(\log 2^{\ell_0}\big)^\eps > 10$
  this case too yields $Q_n > n^{1-\eps}/(K \eps)$.

  It remains to bound $M_n$.  This is easy, since on the event $\Gg$ we
  have $Q_i \geq \frac{i^{1-\eps}}{K \eps}$ for all $i\leq n$.  On this
  event,
  \[
  \sum^{n}_{i=1} Q_i^{-1} \leq
  K \eps \sum^{n}_{i=1} i^{\eps-1} \leq K n^\eps
  \]
  Since $\Gg$ has probability bigger than $\frac56$, we get that the
  $\frac56$-median of the sum, namely $M_n$ is at most $K
  n^\eps$. This completes the induction and shows that under $\Gg$,
  \eqref{eq:MQinfuct} holds for all $n$. Since $\eps$ was
  arbitrary, the lemma is proved.
\end{proof}

\begin{proof}[Proof of Theorem~\ref{T:linearcap}]
  By monotonicity we may consider only $n=2^\ell$. By Lemmas
  \ref{L:linearcap} and \ref{L:highcap},
  $\P(\capa(A_n)>n^{1-o(1)})>1-Cn^{-c}$. By Borel-Cantelli, this
  event holds for all but finitely many such $n$. 
\end{proof}

\subsection*{Acknowledgements}
We would like to thank Itai Benjamini for initiating this project and
for many enlightening discussions. GA and GK's work partially supported by
the Israel Science Foundation. OA's work partially supported by the
NSERC and the ENS.

\begin{flushright}
\footnotesize
Gideon Amir\\
\nolinkurl{gidi.amir@gmail.com}\\
Bar-Ilan University, Ramat Gan, Israel

\medskip

Omer Angel\\
\nolinkurl{angel@math.ubc.ca}\\
University of British Columbia, Vancouver, Canada

\medskip

Gady Kozma\\
\nolinkurl{gady.kozma@weizmann.ac.il}\\
The Weizmann Institute of Science, Rehovot, Israel
\end{flushright}

\end{document}